\documentclass[11pt]{article}
\usepackage{fullpage}
\usepackage{pdflscape}
\usepackage{framed}
\usepackage{mdframed}
\usepackage{enumerate}
\usepackage[inline]{enumitem}
\usepackage[T1]{fontenc}
\usepackage{moresize}
\usepackage{bm}
\usepackage{dsfont}
\usepackage{amsmath}
\usepackage{amssymb}
\usepackage{amsthm}
\usepackage{amsfonts}
\usepackage{stmaryrd}
\usepackage{array}
\usepackage{mathrsfs}
\usepackage{mathtools} % \prescript
\usepackage{extarrows}
\usepackage{stackrel}
\usepackage[nodisplayskipstretch]{setspace}
\usepackage{color}
\usepackage[usenames,dvipsnames]{xcolor}
\usepackage{cancel}
\usepackage{soul}
\usepackage{undertilde}
\usepackage{xfrac}
\usepackage{siunitx}
\usepackage{graphicx}
\usepackage{float}
\usepackage{rotating}
\usepackage{subcaption}
\usepackage{overpic}
\usepackage[all]{xy}
\usepackage{tikz}
\usetikzlibrary{patterns}
\usepackage{booktabs}
\usepackage{dcolumn}
\usepackage{multirow}
\usepackage{diagbox}
\usepackage{verbatim}
\usepackage{listings}
\usepackage{algorithm}
\usepackage{algpseudocode}
\usepackage{fancyvrb}
\usepackage{hyperref}
\usepackage[round]{natbib}
\usepackage{sectsty}

\hypersetup{
    bookmarks=true,         % show bookmarks bar?
    unicode=false,          % non-Latin characters in Acrobat’s bookmarks
    pdftoolbar=true,        % show Acrobat’s toolbar?
    pdfmenubar=true,        % show Acrobat’s menu?
    pdffitwindow=false,     % window fit to page when opened
    pdfstartview={FitH},    % fits the width of the page to the window
    pdftitle={My title},    % title
    pdfauthor={Author},     % author
    pdfsubject={Subject},   % subject of the document
    pdfcreator={Creator},   % creator of the document
    pdfproducer={Producer}, % producer of the document
    pdfkeywords={key1, key2}, % list of keywords
    pdfnewwindow=true,      % links in new PDF window
    colorlinks=true,        % false: boxed links; true: colored links
    linkcolor=blue,         % color of internal links (change box color with linkbordercolor)
    citecolor=blue,         % color of links to bibliography
    filecolor=blue,         % color of file links
    urlcolor=cyan           % color of external links
}

\usepackage{scalerel,stackengine}

\newcommand{\ind}{I}  % Indicator
\newcommand{\mv}{v}
\newcommand{\mx}{x}
\newcommand{\my}{y}
\newcommand{\mX}{X}
\newcommand{\mY}{Y}
\newcommand{\cI}{\mathcal I}

\newcommand{\cF}{\mathcal F}

\newcommand{\cT}{\mathcal T}
\newcommand{\cX}{\mathcal X}
\newcommand{\fA}{A}
\newcommand{\fI}{I}

\newcommand{\mc}{\:\!\mathrm{c}}
\newcommand{\ac}{\:\!\mathrm{ac}}
\newcommand{\cFD}{\cF^{\ac}}

\def\pr{{\mathrm{pr}}}
\def\E{{  E}}
\def\T{{ \mathrm{\scriptscriptstyle T} }}
\def\d{{\mathrm d}}

\newcommand{\sfi}{\mathsf{i}}

\newcommand{\cov}{\mathrm{cov}}
\newcommand{\Cov}{\mathrm{cov}}
\newcommand{\var}{\mathrm{var}}
\newcommand{\Var}{\mathrm{var}}
\newcommand{\TV}{ {TV}}

\newcommand{\HL}{ {H\!L}}

\newcommand\yestag{\addtocounter{equation}{1}\tag{\theequation}}
\newcolumntype{L}[1]{>{\raggedright\let\newline\\\arraybackslash\hspace{0pt}}m{#1}}
\newcolumntype{C}[1]{>{  \centering\let\newline\\\arraybackslash\hspace{0pt}}m{#1}}
\newcolumntype{R}[1]{>{ \raggedleft\let\newline\\\arraybackslash\hspace{0pt}}m{#1}}
\newcolumntype{d}[1]{D{.}{.}{#1}}
\newcolumntype{H}{>{\setbox0=\hbox\bgroup}c<{\egroup}@{}}
\newcolumntype{Z}{>{\setbox0=\hbox\bgroup}c<{\egroup}@{\hspace*{-\tabcolsep}}}

\newtheorem{theorem}{Theorem}%[section]
\newtheorem{proposition}{Proposition}%[section]
\newtheorem{assumption}{Assumption}%[section]
\theoremstyle{definition}
\newtheorem{definition}{Definition}%[section]
\newtheorem{example}{Example}%[section]
\newtheorem{remark}{Remark}%[section]
\let\hat\widehat
\let\tilde\widetilde

\newcommand{\nb}[1]{{\color{black}{#1}}}

\allowdisplaybreaks

\begin{document}

\setlength{\abovedisplayskip}{5pt}
\setlength{\belowdisplayskip}{5pt}
\setlength{\abovedisplayshortskip}{5pt}
\setlength{\belowdisplayshortskip}{5pt}
\hypersetup{colorlinks,breaklinks,urlcolor=blue,linkcolor=blue}

\title{\LARGE On the power of Chatterjee's rank correlation}

\author{
Hongjian Shi\thanks{Department of Statistics, University of Washington, Seattle, WA 98195, USA; e-mail: {\tt hongshi@uw.edu}},~~
Mathias Drton\thanks{Department of Mathematics, Technical University
  of Munich, 85748 Garching b.~M\"unchen, Germany; e-mail: {\tt mathias.drton@tum.de}},~~and~
Fang Han\thanks{Department of Statistics, University of Washington, Seattle, WA 98195, USA; e-mail: {\tt fanghan@uw.edu}}.
}

\date{}

\maketitle

%\vspace{-2em}

\begin{abstract}
  \citet{chatterjee2020new} introduced a simple new rank correlation
  coefficient that has attracted much recent attention.  The
  coefficient has the unusual appeal that it not only estimates a
  population quantity {first proposed by \citet{MR3024030}} 
  that is zero if and only if the underlying pair
  of random variables is independent, but also is
  asymptotically normal under independence.  This paper compares
  Chatterjee's new {correlation} coefficient to three established rank correlations
  that also facilitate consistent tests of independence, namely,
  Hoeffding's $D$, Blum--Kiefer--Rosenblatt's $R$, and
  Bergsma--Dassios--Yanagimoto's $\tau^*$.  We contrast their
  computational efficiency in light of recent advances, and
  investigate their power against local {rotation and mixture} alternatives.  
  Our main results show that Chatterjee's coefficient is unfortunately rate
  sub-optimal compared to $D$, $R$, and $\tau^*$. { The situation is
  more subtle for a related earlier estimator of
  \citet{MR3024030}.}   
  These results favor $D$, $R$, and $\tau^*$ over
  Chatterjee's new {correlation} coefficient for the purpose of testing
  independence.
\end{abstract}

{\bf Keywords:} 
Dependence measure; Independence test; Le~Cam's third
lemma; Rank correlation; Rate-optimality.

\section{Introduction}\label{sec:intro}

%\fbox{revise the whole paper based on the rule ``$[\{($''}

Let $X^{(1)},X^{(2)}$ be two real-valued random variables 
defined on %the same (otherwise unspecified) probability space.
a common probability space.
%and neither is almost surely a constant. 
{We will be} %This paper is
concerned with testing the null hypothesis
\begin{align}\label{eq:H0}
H_0: X^{(1)}\text{ and }X^{(2)}\text{ are independent},
\end{align}
based on a sample from the joint distribution
% realizations $(X^{(1)}_1,X^{(2)}_1),\ldots,(X^{(1)}_n,X^{(2)}_n)$
of $(X^{(1)},X^{(2)})$.
This classical problem has seen
% received
% much attention, with many early contributions but
%also
revived interest in recent years as independence tests constitute a
key component in modern statistical methodology such as, e.g., methods
for causal discovery \citep[Section~18.6.3]{MR3889064}.

% In this paper we will focus on the family that both $X^{(1)}$ and $X^{(2)}$ are (absolutely) continuous. 
%Early attempts
%can be traced back to 
%since \citet{Pearson1895}.

{The problem of testing independence has been examined from a
  number of different perspectives; see, for example, the work of
\citet{meynaoui2019adaptive}, \cite{berrett2020optimal}, and
\cite{kim2020minimax}, and the references therein.  In this paper, our} 
 focus will be on testing $H_0$ via rank correlations that measure ordinal
association. Rank correlations are
particularly attractive for continuous distributions for which 
they are distribution-free under $H_0$.  Early
proposals of rank correlations include the widely-used $\rho$ of
\cite{Spearman1904} and $\tau$ of \cite{Kendall1938}, but also the footrule %$\phi$
of \cite{Spearman1906}, the $\gamma$
of \cite{gini1914ammontare}, and the $\beta$ of \cite{MR39190}.
% among many others.
%A detailed historical review can be found in 
%\citet{MR100941} and \citet{MR2738876}.
Unfortunately,
all five of these rank correlations fail to give a
consistent test of independence.  Indeed, each correlation coefficient consistently estimates a population correlation measure
that takes the same value under $H_0$ and certain fixed alternatives to $H_0$.  This fact leads to 
trivial power at such alternatives.
%cannot be rejected even with an infinite sample. 
%For example, all the five rank correlations (on the population level) are zero
%or $X^{(2)} = (X^{(1)})^2$ and $X^{(1)}$ symmetric about $0$. 
%In comparison,

% To fix it, a consistent correlation measure, which is 0 if and only if the random variables $X^{(1)}$ and $X^{(2)}$ are independent, is needed. 
In order to arrive at a consistent test of independence,
\citet{MR0029139} proposed a correlation measure that, for
absolutely continuous bivariate distributions, vanishes if and only if
$H_0$ holds.
% U-statistic estimator.
\citet{MR0125690} considered a modification that
is
consistent against all dependent bivariate alternatives 
\citep[cf.][]{MR0004426}.   \citet{MR3178526} proposed a
new test of independence and showed its consistency for bivariate
distributions that are discrete, absolutely continuous, or a mixture
of both types.  As pointed out by \cite{MR4185806}, mere
continuity of the marginal distribution functions is sufficient for
consistency of their test.  This follows from a relation discovered by
\citet{yanagimoto1970measures} who implicitly considers the correlation of \citet{MR3178526} when proving a
conjecture of \citet{MR0029139}.

All three aforementioned correlation measures admit natural efficient
estimators in the form of U-statistics that depend only on ranks.
However, in each case, the U-statistic is degenerate and has a
non-normal asymptotic distribution under $H_0$.  In light of this
fact, it is interesting that \citet{MR3024030} were able to construct
a consistent correlation measure $\xi$ 
{which is also able to detect perfect functional dependence}
(see also \citealp{MR3784121}) %with a  rank-based estimator.  
%that is asymptotically normal under the null.  
{and in} a recent paper that received much attention,
\citet{chatterjee2020new} gives a very simple rank correlation, with no tuning parameter involved, that 
surprisingly estimates $\xi$ and {has} an asymptotically normal
null distribution.

This paper  compares Chatterjee's and also
Dette--Siburg--Stoimenov's rank correlation coefficients to the three
obvious competitors given by the $D$ of \citet{MR0029139}, the $R$ of
\citet{MR0125690}, and the $\tau^*$ of \citet{MR3178526}.  Our
comparison considers  three  criteria:
\begin{enumerate}[label=(\roman*)]
\item {\it Statistical consistency of the independence test.}  A correlation measure $\mu$
  assigns to each joint distribution of $(\mX^{(1)},\mX^{(2)})$ a real
  number $\mu(\mX^{(1)},\mX^{(2)})$.  Such a correlation measure is
  consistent in a family of distributions $\cF$ if for all pairs
  $(\mX^{(1)},\mX^{(2)})$ with joint distribution in $\cF$, it holds that
  $\mu(\mX^{(1)},\mX^{(2)})=0$ if and only if $\mX^{(1)}$ is
  independent of $\mX^{(2)}$. %; compare \citet{MR3842884}.
Correlation measures that are consistent within a large nonparametric
family are able to detect non-linear, non-monotone relationship, and facilitate consistent tests of independence.
If a correlation measure $\mu$ is consistent, then the consistency of
tests of independence based on an estimator $\mu_n$ of $\mu$ is
guaranteed by the consistency of that estimator.
\item {\it Computational efficiency.}  Computing ranks requires
  $O(n\log n)$ time.  With a view towards large-scale applications, we
  prioritize rank correlation coefficients that are computable
  without much additional effort, that is, also in
  $O(n\log n)$ time.  This is easily seen to be the case
  for Chatterjee's coefficient but, as we shall survey in
  Section~\ref{sec:prelim}, recent advances clarify that $D$, $R$, and
  $\tau^*$ can be computed similarly efficiently.
\item {\it Statistical efficiency of the independence test.}
  Our final criterion is optimal efficiency in the statistical sense
  \citep[Section~5.4]{MR1335235}.  To assess this, we use different
  local alternatives inspired from work of \citet{MR79384} and of
  \citet{MR119312,MR124108}; the latter type of alternatives was
  further developed in \citet{MR3466185}.  We then call an
  independence test \emph{rate-optimal} (or \emph{rate sub-optimal}) against a family
  of local alternatives if within this family the test achieves the
  detection boundary up to constants  (or not).
\end{enumerate}

The main contribution of this paper pertains to statistical
efficiency.  { Chatterjee's derivation of asymptotic normality for his
rank correlation coefficient relies on a reformulation of his statistic and then invoking a type of permutation central limit theorem that was established in \citet{chao_bai_liang_1993}. We found that a direct use of this technique to analyse the local power is hard.} In recent
related work we were able to overcome a similar issue in a related
multivariate setting \citep{shi2019distribution,deb2019multivariate}
by developing a suitable H\'ajek representation theory
\citep{shi2020rate}. Applying this philosophy here, we { build a particular form of the projected statistic that was introduced in \citet{MR1378827}} to provide an alternative proof of Theorem~2.1 in
\citet{chatterjee2020new} that gives an asymptotic representation. 
Integrating the representation into Le Cam's third lemma and employing
further a version of {the} conditional multiplier central limit theorem 
(cf. Chapter~2.9 in \citealp{MR1385671}), we are then
able to show that the test based on Chatterjee's rank correlation
coefficient is in fact \emph{rate sub-optimal} against the {two}
considered local alternative families; recall point (iii) above.  Our theoretical analysis thus echos Chatterjee's empirical observation,
that is, his test of independence can suffer from low power; 
see Remark~\ref{rmk:morealter} below.
%Furthermore, we are able to show that this rate sub-optimality also
%occurs for the test based on Dette--Siburg--Stoimenov's coefficient,
%but interestingly for this test the specific form of the alternative
%considered here matters.  
In contrast, the tests based on the more
established coefficients $D$, $R$, and $\tau^*$ are all rate-optimal
for all considered local alternative families.  We therefore consider the latter
more suitable for testing independence
{than Chatterjee's test. On the other hand, the test based on Dette--Siburg--Stoimenov's coefficient is empirically observed to have non-trivial power against certain alternatives in finite-sample simulations. A theoretical study of this phenomenon, however, has to be left to the future due to involved technical difficulties.}
%
%Our analysis of statistical efficiency is presented in
%Section~\ref{sec:subopt}.  This analysis is preceded in Section~\ref{sec:prelim}
%by a discussion of the consistency of the different tests we consider as
%well as a survey of recent advances that facilitate efficient
%computation.  Our focus in Sections~\ref{sec:prelim}--\ref{sec:subopt}
%is on continuous distributions, but we present some results of cases
%where ties among ranks may exist nonvanishing probability in
%Section~\ref{sec:newtie}. 
%{Numerical results of the finite-sample performances of all the tests
%are presented in Section~\ref{sec:sim}.}
%A summary of the whole paper, 
%along with some other discussions, is given in Section~\ref{sec:dis}.
The proofs of our claims, including details on examples, are given in 
{the supplementary material}. %Section~\ref{sec:proofs}.

% To complete the story, this paper develops further a complementary comparison, illustrating advantages and disadvantages between Chatterjee's, Dette--Siburg--Stoimenov's, Hoeffding's, Blum--Kiefer--Rosenblatt's, and Bergsma--Dassios--Yanagimoto's rank correlation coefficients, with regard to both continuous and discontinuous cases. 
% In the  continuous case, we claim that all the rank correlation coefficients except Dette--Siburg--Stoimenov's one can be computed in nearly linear time,  
% and all the rank correlation measures perform equally well regarding  consistency, except for Hoeffding's~$D$ that requires joint absolute continuity to achieve consistency. 
% %On the other hand, when ties exist with a nonvanishing probability, Chatterjee's rank correlation coefficient has its own advantages in terms of computational efficiency and consistency. 
% Further discussions on the properties of rank correlations when ties exist with a nonvanishing probability are put in Section~\ref{sec:dis}. 

As we were completing the manuscript, we became aware of independent
work by \citet{cao2020correlations}, who accomplished a similar local
power analysis for Chatterjee's %and Dette--Siburg--Stoimenov's rank
correlation coefficient and presented {a} result that {is} similar to our
Theorem~\ref{thm:power}, Claim~\eqref{eq:subopt}. %and \eqref{eq:subopt-dette}. 
The local alternatives considered in their
paper are, however, different from ours. 
%For instance,
%Dette--Siburg--Stoimenov's rank correlation coefficient cannot
%provide a test of power differentiating the null from the local
%alternatives considered in \citet{cao2020correlations}, while such an
%alternative (family (C) in Section~\ref{sec:subopt})
%introduced in \citet[Sec.~3]{MR3466185} %where Siburg--Stoimenov's rank correlation coefficient is powerful was identified in our results. 
In addition, the two papers differ in
their focus.  The work of \citeauthor{cao2020correlations} concentrates on correlation
measures that are 1 if and only if one variable is a
shape-restricted function of the other variable, while our interest
is in comparing consistent tests of independence.

\section{Rank correlations and independence tests}\label{sec:prelim}

\subsection{Considered rank correlations and their computation}

When considering correlations, we will use the term \emph{correlation
  measure} to refer to population quantities, which we write using
Greek or Latin letters.  The term
\emph{correlation coefficient} is reserved for sample
quantities, which are written with an added subscript $n$.
% the will
% are used to represent population quantities (\emph{correlation
%   measures}), and we use the corresponding ones with a subscript $n$
% to represent the sample analogues (\emph{correlation coefficients}).
% %We use the terms \emph{correlation coefficient} and \emph{correlation} in short to refer to either correlation measure or statistic. 
The symbol $F$ denotes a joint bivariate distribution function for
the considered pair of random variables  $(X^{(1)},X^{(2)})$, and
$F_1$ and $F_2$ are the respective marginal distribution functions.
Throughout,
%we assume that neither $X^{(1)}$ nor $X^{(2)}$ is  almost surely constant, 
$(X^{(1)}_1,X^{(2)}_1),\dots,(X^{(1)}_n,X^{(2)}_n)$ is a sample
comprised of $n$ independent copies of $(X^{(1)},X^{(2)})$.

We now introduce in precise terms the five types of rank correlations
we consider in this paper.  We begin by specifying the correlation
measure and coefficients from \citet{chatterjee2020new} and
\citet{MR3024030}.  To this end, let
$(X^{(1)}_{[1]},X^{(2)}_{[1]}),\dots,(X^{(1)}_{[n]},X^{(2)}_{[n]})$ be
a rearrangement of the sample such that
$X^{(1)}_{[1]}\le\cdots\le X^{(1)}_{[n]}$, with ties, if existing,
broken at random.  Define
\begin{equation}
  \label{eq:r[i]}
  r_{[i]}\equiv\sum_{j=1}^{n}\ind\Big(X^{(2)}_{[j]}\le
  X^{(2)}_{[i]}\Big)
\end{equation}
 %is the rank of $X^{(2)}_{[i]}$ among $X^{(2)}_{[1]},\dots,X^{(2)}_{[n]}$,
with $I(\cdot)$ representing the indicator function, and
$\ell_{[i]}\equiv\sum_{j=1}^{n}\ind(X^{(2)}_{[j]}\ge X^{(2)}_{[i]})$.
We emphasize that if $F_2$ is continuous, then there are almost surely
no ties among $X^{(2)}_{1},\dots,X^{(2)}_{n}$, in which case $r_{[i]}$
is simply the rank of $X^{(2)}_{[i]}$ among
$X^{(2)}_{[1]},\dots,X^{(2)}_{[n]}$.

\begin{definition}%[Chatterjee's $\xi_n$]
\label{eg:cha}
The correlation coefficient of \citet{chatterjee2020new} is
\begin{equation}
  \label{eq:chatt:ties}
\xi_n\equiv1-\frac{n\sum_{i=1}^{n-1}\lvert r_{[i+1]}-r_{[i]}\rvert}{2\sum_{i=1}^{n}\ell_{[i]}(n-\ell_{[i]})}.
\end{equation}
If there are no ties among $X^{(2)}_{1},\dots,X^{(2)}_{n}$, it holds
that
\[
\xi_n =1-\frac{3\sum_{i=1}^{n-1}\lvert r_{[i+1]}-r_{[i]}\rvert}{n^2-1}.
\]
\citet{chatterjee2020new} proved that $\xi_n$ estimates the correlation measure
\[\xi\equiv\frac{\int\var[\E\{\ind(X^{(2)}\ge x)\mid X^{(1)}\}] \d F_2(x)}{\int\var\{\ind(X^{(2)}\ge x)\} \d F_2(x)}.\]
This measure was in fact first proposed in \citet{MR3024030}; cf.~$r(X,Y)$ in their
Theorem~2.  We thus term $\xi$ the
Dette--Siburg--Stoimenov's rank correlation measure.
\end{definition}

We note that $\xi$ was also considered by
\citet{MR3784121}; see the Cram\'er--von Mises index $S^{v}_{2,CVM}$
before their Properties
3.2.  For estimation of $\xi$, \citet{MR3024030} proposed the
following coefficient; denoted $\widehat r_n$ in their Equation~(15).

\begin{definition}%[Dette--Siburg--Stoimenov's $\xi_n^*$]
  \label{def:dette}
  Let $K$ be a symmetric and twice continuously differentiable kernel
  with compact support, and let
  $\overline K(x)\equiv \int_{-\infty}^{x}K(t)\d t$.  Let $h_1, h_2>
  0$ be
  bandwidths that are chosen such that they tend to zero with
  \[nh_1^3\to\infty, ~~~
    nh_1^4\to 0, ~~~
    nh_2^4\to 0, ~~~
    nh_1h_2\to\infty
  \yestag\label{eq:tuning}\]
  as $n\to\infty$.  Define
\begin{align}\label{eq:dette1}
\zeta_n\big(u^{(1)},u^{(2)}\big)\equiv \frac{1}{nh_1}\sum_{i=1}^{n}K\Big(\frac{u^{(1)}-i/n}{h_1}\Big)\overline K\Big(\frac{u^{(2)}-r_{[i]}/n}{h_2}\Big)
\end{align}
with $r_{[i]}$ as in~\eqref{eq:r[i]}.  Then the {Dette--Siburg--Stoimenov's}
 correlation coefficient is
\[\xi_n^{*}\equiv 6\int_{0}^{1}\int_{0}^{1}\big\{\zeta_n\big(u^{(1)},u^{(2)}\big)\big\}^2\d u^{(1)}\d u^{(2)} -2.\]
\end{definition}

%In comparison to Chatterjee's rank correlation, we consider the following ones:
%Hoeffding's $D$ \citep{MR0029139}, 
%Blum--Kiefer--Rosenblatt's $R$ \citep{MR0125690}, 
%and Bergsma--Dassios--Yanagimoto's $\tau^*$ \citep{MR3178526,yanagimoto1970measures}.

Next we introduce two classical rank correlations of
\cite{MR0029139} and \cite{MR0125690}, both of which assess dependence
in a very intuitive way by integrating squared deviations between the
joint distribution function and the product of the marginal
distribution functions. 

\begin{definition}%[Hoeffding's $D$]
\label{eg:hd}
Hoeffding's correlation measure is defined as
\[
  D\equiv\int
  \Big\{F\big(x^{(1)},x^{(2)}\big)-F_1\big(x^{(1)}\big)F_2\big(x^{(2)}\big)\Big\}^2
  \d F\big(x^{(1)},x^{(2)}\big).
\]
It is unbiasedly estimated by the correlation coefficient 
\begin{align*}
 D_n\equiv\;&\frac{1}{n(n-1)\cdots(n-4)}\sum_{i_1\ne \dots\ne i_5}\frac14\\
&\!\!\!\!\!\!\Big[\Big\{\ind\Big(X^{(1)}_{i_1}\leq X^{(1)}_{i_5}\Big)-\ind\Big(X^{(1)}_{i_2}\leq X^{(1)}_{i_5}\Big)\Big\}
         \Big\{\ind\Big(X^{(1)}_{i_3}\leq X^{(1)}_{i_5}\Big)-\ind\Big(X^{(1)}_{i_4}\leq X^{(1)}_{i_5}\Big)\Big\}\Big]\\
&\!\!\!\!\!\!\Big[\Big\{\ind\Big(X^{(2)}_{i_1}\leq X^{(2)}_{i_5}\Big)-\ind\Big(X^{(2)}_{i_2}\leq X^{(2)}_{i_5}\Big)\Big\}
         \Big\{\ind\Big(X^{(2)}_{i_3}\leq X^{(2)}_{i_5}\Big)-\ind\Big(X^{(2)}_{i_4}\leq X^{(2)}_{i_5}\Big)\Big\}\Big],
\yestag\label{eq:Dn}
\end{align*}
which is a rank-based U-statistic of order $5$. 
\end{definition}

\begin{definition}%[Blum--Kiefer--Rosenblatt's $R$]
\label{eg:bkr}
Blum--Kiefer--Rosenblatt's correlation measure is defined as
\[
R\equiv\int \Big\{F\big(x^{(1)},x^{(2)}\big)-F_1\big(x^{(1)}\big)F_2\big(x^{(2)}\big)\Big\}^2 
\d F_1\big(x^{(1)}\big)\d F_2\big(x^{(2)}\big).
\]
It is unbiasedly estimated by the {Blum--Kiefer--Rosenblatt's} correlation coefficient 
\begin{align*}
 R_n\equiv\;&\frac{1}{n(n-1)\cdots(n-5)}\sum_{i_1\ne \dots\ne i_6}\frac14\\
&\!\!\!\!\!\!\Big[\Big\{\ind\Big(X^{(1)}_{i_1}\leq X^{(1)}_{i_5}\Big)-\ind\Big(X^{(1)}_{i_2}\leq X^{(1)}_{i_5}\Big)\Big\}
         \Big\{\ind\Big(X^{(1)}_{i_3}\leq X^{(1)}_{i_5}\Big)-\ind\Big(X^{(1)}_{i_4}\leq X^{(1)}_{i_5}\Big)\Big\}\Big]\\
&\!\!\!\!\!\!\Big[\Big\{\ind\Big(X^{(2)}_{i_1}\leq X^{(2)}_{i_6}\Big)-\ind\Big(X^{(2)}_{i_2}\leq X^{(2)}_{i_6}\Big)\Big\}
         \Big\{\ind\Big(X^{(2)}_{i_3}\leq X^{(2)}_{i_6}\Big)-\ind\Big(X^{(2)}_{i_4}\leq X^{(2)}_{i_6}\Big)\Big\}\Big],
\yestag\label{eq:Rn}
\end{align*}
which is a rank-based U-statistic of order $6$. 
\end{definition}

More recently, \citet{MR3178526} introduced the following rank
correlation, which is connected to work by
\citet{yanagimoto1970measures}.  We refer the reader to
\citet{MR3178526} for a motivation via con-/disconcordance of 4-point patterns and connections to Kendall's tau.

\begin{definition}%[Bergsma--Dassios--Yanagimoto's $\tau^*$]
  \label{eg:bdy}
  Write $\ind(x_1,x_2<x_3,x_4)\equiv
  \ind(\max\{x_1,x_2\}<\min\{x_3,x_4\})$.  The
  Bergsma--Dassios--Yanagimoto's correlation measure is
  \begin{align*}
\tau^*\equiv\;
& 4\pr\Big(X^{(1)}_{1},X^{(1)}_{3}< X^{(1)}_{2},X^{(1)}_{4}\ ,\ 
           X^{(2)}_{1},X^{(2)}_{3}< X^{(2)}_{2},X^{(2)}_{4}\Big)\\
&+4\pr\Big(X^{(1)}_{1},X^{(1)}_{3}< X^{(1)}_{2},X^{(1)}_{4}\ ,\ 
           X^{(2)}_{2},X^{(2)}_{4}< X^{(2)}_{1},X^{(2)}_{3}\Big)\\
&-8\pr\Big(X^{(1)}_{1},X^{(1)}_{3}< X^{(1)}_{2},X^{(1)}_{4}\ ,\ 
           X^{(2)}_{1},X^{(2)}_{4}< X^{(2)}_{2},X^{(2)}_{3}\Big).
\end{align*}
It is unbiasedly estimated by a U-statistic of order $4$, namely, the
{Bergsma--Dassios--Yanagimoto's} correlation coefficient
\begin{align*}
 \tau^*_n\equiv\;&\frac{1}{n(n-1)(n-2)(n-3)}\sum_{i_1\ne \dots\ne i_4}\\
&   \Big\{\ind\Big(X^{(1)}_{i_1},X^{(1)}_{i_3}< X^{(1)}_{i_2},X^{(1)}_{i_4}\Big)
         +\ind\Big(X^{(1)}_{i_2},X^{(1)}_{i_4}< X^{(1)}_{i_1},X^{(1)}_{i_3}\Big)\\
&        -\ind\Big(X^{(1)}_{i_1},X^{(1)}_{i_4}< X^{(1)}_{i_2},X^{(1)}_{i_3}\Big)
         -\ind\Big(X^{(1)}_{i_2},X^{(1)}_{i_3}< X^{(1)}_{i_1},X^{(1)}_{i_4}\Big)\Big\}\\
&   \Big\{\ind\Big(X^{(2)}_{i_1},X^{(2)}_{i_3}< X^{(2)}_{i_2},X^{(2)}_{i_4}\Big)
         +\ind\Big(X^{(2)}_{i_2},X^{(2)}_{i_4}< X^{(2)}_{i_1},X^{(2)}_{i_3}\Big)\\
&        -\ind\Big(X^{(2)}_{i_1},X^{(2)}_{i_4}< X^{(2)}_{i_2},X^{(2)}_{i_3}\Big)
         -\ind\Big(X^{(2)}_{i_2},X^{(2)}_{i_3}< X^{(2)}_{i_1},X^{(2)}_{i_4}\Big)\Big\}.
\yestag\label{eq:taustarn}
\end{align*}
\end{definition}

\begin{remark}[Relation between $D_n$, $R_n$, and $\tau^*_n$] \label{remark:identity}
As conveyed by Equation~(6.1) in \cite{MR4185806}, as long as $n\ge 6$ and there are
no ties in the data, it holds that $12D_n+24R_n=\tau^*_n$.
Consequently, $12D+24R=\tau^*$ given continuity but not necessarily
absolute continuity of $F$; compare page 62 of \citet{yanagimoto1970measures}.
%In particular, the fastest algorithm to compute $R_n$ (Proposition~\ref{prop:corr-compu}(iii)) is
%combing algorithms for computing $D_n$ and $\tau^*_n$
%and using the first identity, and 
%Proposition~\ref{prop:d-consistent}(iv) is established using the properties of $D,R$ and the latter identity. 
\end{remark}

At first sight the computation of the different correlation
coefficients appears to be of very different complexity.  However,
this is not the case due to recent developments, which yield nearly
linear computation time for all
coefficients except $\xi^*_n$.

\begin{proposition}[Computational efficiency]
  \label{prop:corr-compu}
  If data have no ties, then $\xi_n$, $D_n$, $R_n$, and $\tau^*_n$ can
  all be computed in $O(n\log n)$ time.
\end{proposition}
\begin{proof}
  It is evident from its simple form that  $\xi_n$ can be computed in
  $O(n\log n)$ time \citep[Remark~4]{chatterjee2020new}.  The result
  about $D_n$ is due to \citet[Section~5]{MR0029139}; see also
  \citet[page~557]{MR3842884}.  The claim about $\tau^*_n$ is based on
  recent new methods due to \citet[Corollary~4]{doi:10.1137/1.9781611976465.136} 
  and \citet[Theorem~6.1]{even2020independence}; for
  an implementation see \citet{R:independence}.
  The claim about $R_n$ then follows from the relation given in
  Remark~\ref{remark:identity}.
\end{proof}

\begin{remark}[Computation of $\xi^*_n$]
  The definition of $\xi^*_n$ involves an integral over the unit
  square $[0,1]^2$.  How quickly the integral can be computed depends
  on smoothness properties of the considered kernel and the bandwidth
  choice.  \citet[Remark~5]{chatterjee2020new} suggests a time
  complexity of $O(n^{5/3})$.  Indeed, for a symmetric and four times
  continuously differentiable kernel $K$ that has compact support,
  there is a choice of bandwidths $h_1,h_2$ that satisfies the
  requirements of Definition~\ref{def:dette} and for which $\xi_n^*$
  can be approximated with an absolute error of order $o(n^{-1/2})$ in
  $O(n^{5/3})$ time.

  To accomplish this we may choose $h_1=h_2=n^{-1/4-\epsilon}$ for
  small $\epsilon>0$ and apply Simpson's rule to the two-dimensional
  integral in the definition of $\xi^*_n$.  By assumptions on $K$, the
  function $\zeta_n^2$ has continuous and compactly supported fourth
  partial derivatives that are bounded by a constant multiple of
  $h_1^{-5}$.  The error of Simpson's rule applied with a grid of
  $M^2$ points in $[0,1]^2$ is then $O(h_1^{-5}/M^4)$.  With
  $M^2=O(h_1^{-5/2}n^{1/4+\epsilon/2})=O(n^{7/8+3\epsilon})$, this
  error becomes $O(n^{-1/2-\epsilon})=o(n^{-1/2})$.  Due to the
  compact support of $K$, one evaluation of $\zeta_n$ requires $O(nh_1)$
  operations.  The overall computational time is thus
  $O(nh_1M^2)=O(n^{13/8 + 2\epsilon})$, which is $O(n^{5/3})$ as long
  as $\epsilon\le1/48$.
\end{remark}

\begin{remark}[Computation with ties] \label{remark:tie-computation}
  When the data can be considered as generated from a continuous
  distribution but featuring a small number of ties due to rounding,
  then ad-hoc breaking of ties poses little problem.  In contrast, if
  ties arise due to discontinuity of the data-generating distribution,
  then the situation is more subtle.  In this case, Chatterjee's
  $\xi_n$ is to be computed in the form from~\eqref{eq:chatt:ties},
  but the computational time clearly remains $O(n\log n)$. In
  contrast, $\xi_n^*$ is no longer a suitable estimator of
  $\xi$. Hoeffding's formulas for $D_n$ continue to apply with ties,
  keeping the computation at $O(n\log n)$ but, as we shall emphasize in Section \ref{sec:newtie}, 
  the estimated $D$ may lose some of its appeal.  Bergsma--Dassios--Yanagimoto's $\tau^*_n$ is
  suitable also for discrete data, but the available implementations
  that explicitly account for data with ties \citep{R:taustar} are
  based on the $O(n^2\log n)$ algorithm of \citet[Sec.~3]{MR3481807}
  or the slighly more memory intensive but faster $O(n^2)$ algorithm
  of \citet[Sec.~2.2]{heller2016computing}.  Computation of $R_n$ 
  with ties is also $O(n^2)$ \citep{MR3842884,R:taustar}.  
\end{remark}

%\subsection{Rank correlations without ties}\label{sec:notie}
\subsection{Consistency}

In the rest of this section as well as in Section~\ref{sec:subopt}, we
will always assume that the joint distribution function $F$ is
continuous, though not necessarily jointly absolutely continuous with regard to the Lebesgue measure.
%This is equivalent to say, marginal distribution functions $F_1$ and $F_2$ are continuous, though not necessarily marginally absolutely continuous.
Accordingly, both 
$X^{(1)}_1,\dots,X^{(1)}_n$ and $X^{(2)}_1,\dots,X^{(2)}_n$ are free of ties
with probability one. 
To clearly state the following results, we introduce three families of
bivariate distributions specified via their joint distribution
function $F$:
\begin{align*}
%\cF&=\{F: F\text{ is any bivariate distribution}\},\\
\cF^{\mc}&\equiv \bigl\{F: F\text{ is continuous as a bivariate function}\bigr\},\\%=\cF^{\xi}=\cF^{R}=\cF^{\tau^*},\\
%\cF^{\xi}&=\cF^{\mc},\\
%\{F: F\text{ is marginally not almost surely constant}\},\\
\cFD&\equiv \bigl\{F: F\text{ is absolutely continuous with regard to the Lebesgue measure}\bigr\},\\
%\cF^{R}&=\cF^{\mc},\\
%\cF^{\tau^*}&=\cF^{\mc}.
%\{F: F\text{ is discrete, marginally continuous, or a mixture of both}\}.
\cF^{\rm DSS}&\equiv \bigl\{F\in\cF^{\mc}: F\text{ has a copula
               $C(u^{(1)},u^{(2)})$ that  is three and two times continuously }\\
             &\qquad\qquad \text{differentiable
with respect to the arguments $u^{(1)}$ and $u^{(2)}$, respectively}\bigr\}\yestag\label{eq:bifamily1}. 
\end{align*}
Recall that the copula of $F$ satisfies $F(x^{(1)},x^{(2)})=C\{F_1(x^{(1)}),F_2(x^{(2)})\}$.
% Recall that $F$ is the joint distribution function of $X=(X^{(1)},X^{(2)})$, and see Section~\ref{sec:dis} ahead for further results as the continuity requirement is dropped. 

% We first summarize the properties of the correlation measures and coefficients, starting from Property (i), computational efficiency. Noting the assumption of no ties, the following proposition summarizes the computational complexity of calculating each of the aforementioned five rank correlation coefficients; note that in the following all but $\xi_n^*$ can be computed in nearly linear time. 

% \begin{proposition}[Computational efficiency]\label{prop:corr-compu}
% Assuming $F$ to be continuous, one can compute
% \begin{enumerate}[label=(\roman*)]
% \item $\xi_n$    in $O(n\log n)$ time (\citealp[page~2, Remark~4]{chatterjee2020new});
% \item approximated $\xi_n^*$ with the absolute error $o(n^{-1/2})$ in $O(n^{5/3})$ time for any arbitrarily small constant $\delta>0$ assuming $K$ is symmetric with compact support and four times continuously differentiable and $h_1,h_2$ are chosen appropriately (\citealp[page~2, Remark~5]{chatterjee2020new});
% \item $D_n$      in $O(n\log n)$ time (\citealp[Section~5]{MR0029139}, \citealp[page~557]{MR3842884});
% \item $R_n$      in $O(n\log n)$ time (\citealp[Equation~(6.1)]{MR4185806}, \citealp[page~557]{MR3842884}, \citealp[Corollary~4]{doi:10.1137/1.9781611976465.136}, see also Remark \ref{remark:identity});
% \item $\tau^*_n$ in $O(n\log n)$ time \citep[Corollary~4]{doi:10.1137/1.9781611976465.136}.
% \end{enumerate}
% \end{proposition}

We first discuss the large-sample consistency of the correlation coefficients as
estimators of the corresponding correlation measures.  Convergence in
probability is denoted $\longrightarrow_p$.

% The second proposition shows the   consistency of correlation coefficients $\xi_n$, $\xi_n^*$, $D_n$, $R_n$, and $\tau^*_n$ to the corresponding correlation measures. %Recall that we have assumed $F$ to be continuous throughout this section.

\begin{proposition}[Consistency of estimators]\label{prop:strong}
  For any $F\in\cF^{\mc}$ and  $n\to\infty$, we have
  \begin{equation*}
    %\label{eq:inprob}
    \xi_n\longrightarrow_p \xi,
    ~~~D_n\longrightarrow_p D,
    ~~~R_n\longrightarrow_p R, 
    ~~~\text{and}~~~
    \tau^*_n\longrightarrow_p \tau^*.      
  \end{equation*}
If in addition $F\in\cF^{\rm DSS}$ {and $K,h_1,h_2$ satisfy all assumptions stated in Definition~\ref{def:dette}}, then also $\xi^*_n\longrightarrow_p \xi$.
\end{proposition}
\begin{proof}
  The claim about $\xi_n$ is Theorem~1.1 in \citet{chatterjee2020new},
  and the one about $\xi_n^*$ {is proved in the supplement Section~\ref{subsec:proof:prop:strong} based on a revised version of} Theorem~3 in \cite{MR3024030}.  The
  remaining claims are immediate from U-statistics theory (e.g.,
  Proposition~1 in \citealp{MR3842884}, Theorem~5.4.A in
  \citealp{MR595165}). 
\end{proof}

Next, we turn to the correlation measures themselves.  % and their consistency.
% The third proposition shows the consistency of correlation measures
% $\xi$, $D$, $R$, and $\tau^*$
%in view of the discussions for Property (i) in our introduction.
It is clear that $\xi$, $D$, and $R$ are always nonnegative, and that
the same is true for $\tau^*$ when applied to $F\in\cF^{\mc}$; this
follows from Remark~\ref{remark:identity}.  The consistency
properties for continuous observations can be summarized as follows.

\begin{proposition}[Consistency of correlation
  measures]\label{prop:d-consistent}
  Each one of the correlation measures $\xi$, $R$, and $\tau^*$ is consistent for the entire class
  $\cF^{\mc}$, that is, if
  $F\in\cF^{\mc}$, then $\xi=0$ (or $R=0$ or $\tau^*=0$) if
  and only if the pair $(X^{(1)},X^{(2)})$ is independent.
  Hoeffding's $D$ is consistent for $\cF^{\ac}$ but not $\cF^{\mc}$.
\end{proposition}
\begin{proof}
   The consistency of $\xi$
  is Theorem~2 of \cite{MR3024030}, and Theorem~1.1 of
  \cite{chatterjee2020new}.  The consistency of $R$ is shown in detail
  in Theorem 2 of \citet{MR3842884}; see also p.~490 in
  \citet{MR0125690}.  The consistency of $\tau^*$ was established for
  $\cF^{\ac}$ in
  Theorem~1 in \citet{MR3178526}, and that for $\cF^{\mc}$ can be shown
  via Remark~\ref{remark:identity}; compare  Theorem~6.1 of
  \cite{MR4185806}.  Finally, the claim about $D$ follows from
  Theorem~3.1 of \cite{MR0029139} and its generalization in
  Proposition~3 of \cite{yanagimoto1970measures}. 
\end{proof}

\subsection{Independence tests}\label{sec:tests}

For large samples, computationally efficient independence tests may be
implemented using the
asymptotic null distributions of the correlation coefficients, which are summarized below.

% The fourth proposition summarizes some asymptotic results for the five correlation coefficients under $H_0$ in \eqref{eq:H0}. 
% %We will focus on the continuous case, other cases including discrete distributions will be discussed later.
% %The following two propositions summarize the limiting null distributions of 
% %$\xi_n$, $D_n$, $R_n$, and $\tau^*_n$.

\begin{proposition}[Limiting null distributions]\label{prop:cha}
%Let $(X^{(1)}_1,X^{(2)}_1),\dots,(X^{(1)}_n,X^{(2)}_n)$ be $n$ independent copies of $(X^{(1)},X^{(2)})$ where
Suppose $F\in\cF^{\mc}$ has $X^{(1)}$ and $X^{(2)}$ independent. As $n\to\infty$, it holds that
\begin{enumerate}[label=(\roman*)]
\item for Chatterjee's correlation coefficient $\xi_n$, $n^{1/2}\xi_n\to N(0,2/5)$ in distribution (Theorem~2.1 in \citealp{chatterjee2020new}); 
%\item $n^{1/2}\xi_n^{*}\to N(0,64/5)$ in distribution 
%assuming that $F\in\cF^{\rm DSS}$
%and $K,h_1,h_2$ satisfy all assumptions stated in Definition~\ref{def:dette} (Theorem~3 in \citealp{MR3024030}); 
\item for Dette--Siburg--Stoimenov's correlation coefficient $\xi_n^{*}$, 
{$n^{1/2}\xi_n^{*}\to 0$ in probability}
assuming that $F\in\cF^{\rm DSS}$
and $K,h_1,h_2$ satisfy all assumptions stated in
Definition~\ref{def:dette} ({revised version of Theorem~3 in
  \citealp{MR3024030}; see Section \ref{subsec:proof:prop:cha} of
  the supplementary material}); 
\item  for $\mu\in\{D, R, \tau^*\}$, %we have
\[n\mu_n\to\sum_{v_1,v_2=1}^{\infty}\lambda^{\mu}_{v_1,v_2}\Big(\xi_{v_1,v_2}^2-1\Big)~~~\text{in distribution},\]
where 
\[
\lambda^{\mu}_{v_1,v_2}=
\begin{cases}1/(\pi^4 v_1^2 v_2^2) &\text{ when }\mu=D,R,\\
36/(\pi^4 v_1^2 v_2^2) &\text{ when }\mu=\tau^*,\end{cases}
\]
for $v_1,v_2=1,2,\dots$,  
%are the non-zero eigenvalues of the integral equation
%\[
%\E\Big[\binom{m}{2}  h^{\mu}_{2}\Big((x^{(1)},x^{(2)}), (X^{(1)},X^{(2)}); 
%                        \pr_{X^{(1)}}\times \pr_{X^{(2)}}\Big) 
%         \psi\Big((X^{(1)},X^{(2)})\Big)\Big]
%=\lambda \psi\Big((x^{(1)},x^{(2)})\Big),
%\yestag\label{eq:eigen}
%\]
%in which $m$ is the order of U-statistic $\mu_n$ and $X^{(1)}$ is independent of $X^{(2)}$, 
and $\{\xi_{v_1,v_2}\}$ as independent standard normal random variables 
(Proposition~7 in \citealp{MR3842884}, Proposition 3.1 in \citealp{MR4185806}). 
\end{enumerate}
\end{proposition}

For a given significance level $\alpha\in(0,1)$, let $z_{1-\alpha/2}$ be
the $(1-\alpha/2)$-quantile of the standard normal distribution.  Then the
asymptotic test based on Chatterjee's $\xi_n$ is
\begin{align*}
  {T}^{\xi_n}_{\alpha}\equiv \ind\big\{n^{1/2}|\xi_n| >(2/5)^{1/2}\cdot 
                        z_{1-\alpha/2}\big\}. %\quad\text{and} \quad
%  {T}^{\xi_n^*}_{\alpha}\equiv \ind\big\{n^{1/2}|\xi_n^*|
%                          >(64/5)^{1/2}\cdot z_{1-\alpha/2}\big\},
\end{align*}
The tests based on $\mu_n$ with $\mu\in\{D, R, \tau^*\}$ take the form
\[
{T}^{\mu_n}_{\alpha}\equiv \ind\big(n\, \mu_n > q^{\mu}_{1-\alpha}\big),~~~ %
q^{\mu}_{1-\alpha}\equiv \inf\Big[x:\pr\Big\{\sum_{v_1,v_2=1}^{\infty}\lambda^{\mu}_{v_1,v_2}\Big(\xi_{v_1,v_2}^2-1\Big)\le x\Big\}\ge 1-\alpha \Big],
\]
where $\lambda^{\mu}_{v_1,v_2}$ and $\xi_{v_1,v_2}$,
$v_1,v_2=1,\dots,n,\dots$ were presented in Proposition
\ref{prop:cha}.  We note that \cite{R:taustar} gives a routine to
compute the needed quantiles.
{It is unclear how to implement the test based on
  Dette--Siburg--Stoimenov's $\xi_n^*$ without the need for simulation
  or permutation as a non-degenerate limiting null distribution is
  currently unknown.} 

Given the distribution-freeness of ranks for the class $\cF^{\mc}$,
Proposition~\ref{prop:cha} yields uniform asymptotic validity
of the tests just defined.  Moreover,
Propositions~\ref{prop:strong}--\ref{prop:d-consistent} yield
consistency at fixed alternatives.  We summarize these facts below.
% of the tests mAs
% direct corollary of Propositions~\ref{prop:strong}--\ref{prop:cha},
% and using the regarding the uniform validity and consistency of tests
% ${T}^{\xi_n}_{\alpha}$, ${T}^{\xi^*_n}_{\alpha}$,
% ${T}^{D_n}_{\alpha}$, ${T}^{R_n}_{\alpha}$, and
% ${T}^{\tau^*_n}_{\alpha}$ is summarized below.

\begin{proposition}[Uniform validity and consistency of tests]$ $
%Let $(X^{(1)}_1,X^{(2)}_1),\dots,(X^{(1)}_n,X^{(2)}_n)$ be $n$
%independent copies of $(X^{(1)},X^{(2)})\sim \P$.
  The tests based on the correlation coefficients $\mu_n \in \{\xi_n, D_n, R_n, \tau^*_n\}$ are 
  uniformly valid in the sense that
  \[
\lim_{n\to \infty}\sup_{F\in\cF^{\mc}}\pr({T}^{\mu_n}_{\alpha}=1\mid H_0)=\alpha.\yestag\label{eq:valid2}
\]
Moreover, these tests are consistent, i.e., for fixed $F\in\cF^{\mc}$ such that $X^{(1)}$ and $X^{(2)}$ are
dependent and $\mu_n\in\{\xi_n,R_n,\tau^*_n\}$, it holds that
 \[\lim_{n\to \infty}\pr({T}^{\mu_n}_{\alpha}=1\mid H_1)=1.\yestag\label{eq:valid3}\] 
The conclusion \eqref{eq:valid3} holds for $\mu_n=D_n$ if assuming further that $F\in\cFD$. 
\end{proposition}

\section{Local power analysis}\label{sec:subopt}

This section investigates the local power of the {four} rank
correlation-based tests of $H_0$ introduced in Section
\ref{sec:tests}.  To this end, we consider two classical { and well-used}
families of alternatives to the null hypothesis of independence: {\it
  rotation alternatives} ({\it Konijn alternatives};
\citealp{MR79384}) and {\it mixture alternatives} ({\it Farlie-type
  alternatives}; \citealp{MR119312,MR124108}; see also \citealp{MR3466185}). %For the mixture
%alternatives, we distinguis two classes of models that will give rise to different performances of Dette--Siburg--Stoimenov's coefficient. 
%These two alternatives are classical in the context of independence testing. %, 
%provided that the most natural model of alternatives 
%is lacking in the statistical literature \citep[Section~5.4]{MR1335235}. 
%Within these two families, we establish the rate-optimality and rate sub-optimality (the rate of bivariate alternatives here is the usual $n^{1/2}$ rate) of the four rank correlation tests. 
%Other families of bivariate alternatives can be found in \citet{MR2288045}; %and \citet{MR3466185}. 
%The results regarding to rate-optimality and the proofs are quite similar.
%conclusions for them and proofs are similar to the considered ones.

%We first formally introduce the  families of local alternatives. 

\vspace{.2cm}

{(A) Rotation alternatives. }%\label{sec:alter-kon}
%We now investigate the power of the proposed tests from an asymptotic
%perspective. To justify the rate-optimality of the proposed test, 
%We consider the following alternative model introduced by \citet{MR79384}.
%In detail, 
Let $Y^{(1)}$ and $Y^{(2)}$ be two real-valued independent random
variables that have mean zero and are absolutely continuous with Lebesgue-densities $f_1$ and $f_2$, respectively.  
For  $\Delta\in(-1,1)$, consider 
\[
\mX = \bigg(\begin{matrix}X^{(1)} \\ X^{(2)}\end{matrix}\bigg)
\equiv\bigg(\begin{matrix}1 & \Delta\\
                      \Delta & 1\end{matrix}\bigg) 
              \bigg(\begin{matrix}Y^{(1)} \\ Y^{(2)}\end{matrix}\bigg)
= \fA_{\Delta}\bigg(\begin{matrix}Y^{(1)} \\ Y^{(2)}\end{matrix}\bigg)
= \fA_{\Delta}\mY.
\yestag\label{eq:rotamodel}
\]
 %; all the parameters are chosen such that $\fA_{\Delta}^{-1}$
 %exists.
For all $\Delta\in(-1,1)$, the matrix $\fA_{\Delta}$ is clearly full
rank and invertible.
% for all $\Delta\in\Theta\equiv (-\Delta^*,\Delta^*)$ with some (sufficiently small) constant $\Delta^*>0$.  
For any $\Delta\in(-1,1)$, let $f_{\mX}(\mx;\Delta)$ denote the
density of $\mX=\fA_{\Delta}\mY$.  We then make the following assumptions on $Y^{(1)},Y^{(2)}$.

\begin{assumption}\label{asp:kon}
It holds that %$Y^{(1)},Y^{(2)}$ is such that 
\begin{enumerate}[label=(\roman*)]
%\item $Y^{(1)},Y^{(2)}$ are absolutely continuous with density functions $f_{Y^{(1)}},f_{Y^{(2)}}$ and independent;
%\item $\fA_{\Delta}$ is full rank for all $\Delta\in\Theta\equiv (-\Delta^*,\Delta^*)$ for some (sufficiently small) constant $\Delta^*>0$; 
\item\label{asp:kon1} the distributions of $\mX$ have a common support for all $\Delta\in(-1,1)$, so that without loss of generality ${\cX}\equiv \{\mx: f_{\mX}(\mx;\Delta) > 0\}$ is independent of $\Delta$; 
\item\label{asp:kon2} the density $f_k$ is absolutely continuous with non-constant
  logarithmic derivative
  $\rho_k\equiv f'_k/f_k$,
  %,  
%%such that $\E\{\rho_{Y^{(k)}}(Y^{(k)})\}=0$, 
%where $\rho_{Y^{(k)}}(z)\equiv f'_{Y^{(k)}}(z)/f_{Y^{(k)}}(z)$
  $k=1,2$; 
\item\label{asp:kon3} the Fisher information of $\mX$ relative to $\Delta$ at the point $0$, denoted $\cI_{\mX}(0)$, is strictly positive, and $\E\{(Y^{(k)})^2\} <\infty$, $\E[\{\rho_k(Y^{(k)})\}^2] <\infty$ for $k=1,2$.
\end{enumerate}
\end{assumption}

\begin{remark}\label{rmk:kon}
Assumption~\ref{asp:kon}\ref{asp:kon2},\ref{asp:kon3} implies $\E\{\rho_k(Y^{(k)})\}=0$ and $\cI_{\mX}(0)<\infty$. 
\end{remark}

\begin{example}\label{ex:kon}
Suppose $f_k(z)$ is absolutely continuous and positive for all real numbers $z$, $k=1,2$. If 
\[
  \E\big(Y^{(k)}\big) =0,~~~
  \E\big\{\big(Y^{(k)}\big)^2\big\} <\infty, ~~~
  \E\big[\big\{\rho_k\big(Y^{(k)}\big)\big\}^2\big] <\infty,~~~\text{for}~k=1,2,
  \yestag\label{eq:kon1}
\]
% is satisfied, where $\rho_{Y^{(k)}}(z)\equiv
% f'_{Y^{(k)}}(z)/f_{Y^{(k)}}(z)$,
then Assumption~\ref{asp:kon} holds.  As a special case,
Assumption~\ref{asp:kon} holds if $\mY^{(1)}$ and $\mY^{(2)}$ are
centred and follow normal distributions or $t$-distributions with not necessarily integer-valued degrees of
freedom greater than two.
% \mbox{}
% \begin{enumerate}[label=(\alph*)]
% \item\label{ex:kon1} 
%   Let $f_{Y^{(k)}}(z)$ be absolutely continuous and positive for all $z\in\mathbb{R}$, $k=1,2$. If 
%   \[
%   \E\big(Y^{(k)}\big) =0,~~~
%   \E\big\{\big(Y^{(k)}\big)^2\big\} <\infty, ~~~
%   \E\Big[\big\{\rho_{Y^{(k)}}\big(Y^{(k)}\big)\big\}^2\Big] <\infty,~~~\text{for}~k=1,2,
%   \yestag\label{eq:kon1}
%   \]
%   %is satisfied, where $\rho_{Y^{(k)}}(z)\equiv
%   %f'_{Y^{(k)}}(z)/f_{Y^{(k)}}(z)$,
%   then Assumption~\ref{asp:kon} holds.
% \item\label{ex:kon2} In particular, if $\mY^{(1)}$ and $\mY^{(2)}$ are
%   centred normal or follow centred $t$-distributions with
%   degrees of freedom (not necessarily integer-valued) greater than
%   two, then Assumption~\ref{asp:kon} holds.  
% \end{enumerate}
\end{example}

{(B) Mixture alternatives. }%\label{sec:alter-gof}
%Suppose that we now want to test $H_0: F_0 = F_X F_Y$, where $F_0$ is the joint
%distribution function of $(X, Y)$ with the associated marginal distribution functions
%of $X$ and $Y$ being $F_X$ and $F_Y$, respectively, and we consider a sequence of
%local alternatives $H_n: F_n = (1-\gamma/n^{1/2})F_0+(\gamma/n^{1/2})K$ for a fixed 
%$\gamma > 0$ and a fixed joint distribution $K\ne F_0$ with marginals $F_X$ and $F_Y$. 
Consider the following mixture alternatives that were used in
\citet[Sec.~3]{MR3466185}.  Let $F_{1}$ and $F_{2}$ be fixed univariate
distribution functions that are absolutely continuous with
Lebesgue-density functions $f_{1}$ and $f_{2}$, respectively.  Let
$F_0\big(x^{(1)},x^{(2)}\big)=F_{1}\big(x^{(1)}\big)F_{2}\big(x^{(2)}\big)$
be the product distribution function yielding independence, and let
$G\ne F_0$ be a fixed bivariate distribution function that is
absolutely continuous and such that $(X^{(1)},X^{(2)})$ are dependent
under $G$. %with marginals $F_{1}$ and $F_{2}$,
Let the density functions of $F_0$ and $G$, denoted by $f_0$ and $g$,
respectively, be continuous and have compact supports.  Then define the
following alternative model for the distribution of $\mX=(X^{(1)},X^{(2)})$:
\[F_{\mX}\equiv(1-\Delta)F_{0}+\Delta G,\yestag\label{eq:gof-model}\]
with $0\le \Delta\le 1$. 

We make the following additional assumptions on $F_0$ and $G$.
% , where
% we refer to the function $s(x)\equiv g(x)/f_0(x)-1$.

%\begin{align*}
%(1-\Delta)F_1F_2+\Delta G &= [(1-\Delta)F_1+\Delta G_1][(1-\Delta)F_2+\Delta G_2]\\
%                          & \Updownarrow\\
%                 \Delta G &= -\Delta(1-\Delta)F_1F_2+\Delta(1-\Delta)F_1G_2+\Delta(1-\Delta)F_2G_1+\Delta^2 G_1G_2\\
%                          & \Updownarrow\\
%        \Delta (G-G_1G_2) &= -\Delta(1-\Delta)F_1F_2+\Delta(1-\Delta)F_1G_2+\Delta(1-\Delta)F_2G_1-\Delta(1-\Delta) G_1G_2\\
%                          & \Updownarrow\\
%               (G-G_1G_2) &= -(1-\Delta)(F_1-G_1)(F_2-G_2)\\
%\end{align*}

\begin{assumption}\label{asp:far}
It holds that %$F_0,K$ is such that 
\begin{enumerate}[label=(\roman*)]
\item\label{asp:far1} the distribution $G$ is absolutely continuous with respect to
  $F_0$ and $s(x)\equiv g(x)/f_0(x)-1$ is continuous; 
\item\label{asp:far2} the conditional expectation $\E\{s(Y)|Y^{(1)}\}=0$ almost surely for $Y=(Y^{(1)},Y^{(2)})\sim F_0$;
  % ,  
% $Y^{(k)}$ is distributed with distribution function $F_k$ for $k=1,2$,
% $Y^{(1)}$ and $Y^{(2)}$ are independent;
%\item\label{asp:far3} $\E[s(Y)|Y^{(2)}]=0$ almost surely for $Y=(Y^{(1)},Y^{(2)})\sim F_0$; 
\item\label{asp:far4} the function $s$ is not additively separable, i.e., there
  do not exist univariate functions $h_1$ and $h_2$ such that
  $s(x)=h_1(x^{(1)})+h_2(x^{(2)})$;
  %and $h_k(x^{(k)})$ depend only on $x^{(k)}$ for $k=1,2$;
%\item $s(x)$ is a sum of multiplicatively separable functions, i.e., $s(x)=\sum_{v=1}^{d}h_{1,v}(x^{(1)})h_{2,v}(x^{(2)})$, where $h_{k,v}(x^{(k)})$ depends only on $x^{(k)}$ for $k=1,2$, $v=1,\dots,d$, $d<\infty$;
%\item $s(x)=h_{1,0}(x^{(1)})h_{2,0}(x^{(2)})$, where $h_{k,0}(x^{(k)})$ depends only on $x^{(k)}$ and is non-constant for $k=1,2$;
\item\label{asp:far5} the Fisher information $\cI_{\mX}(0)>0$. %and $\E[\{h_{k,v}(Y^{(k)})\}^2]<\infty$ for $k=1,2$, $v=1,\dots,d$.
\end{enumerate}
\end{assumption}

\begin{remark}\label{rmk:far}
In this model, $g(x)/f_0(x)$ is continuous and has compact support, which guarantees that $\cI_{\mX}(0)<\infty$. 
\end{remark}

\begin{example}\label{ex:far}
%\mbox{}
%\begin{enumerate}[label=(\alph*)]
%\item 
(Farlie alternatives)
Let $G$ in \eqref{eq:gof-model} be given as
\[
G\big(x^{(1)},x^{(2)}\big)\equiv F_{1}\big(x^{(1)}\big)F_{2}\big(x^{(2)}\big)\Big[1+\big\{1-F_{1}\big(x^{(1)}\big)\big\}\big\{1-F_{2}\big(x^{(2)}\big)\big\}\Big].
\]
Then Assumption~\ref{asp:far} is satisfied \citep{MR81575,MR99728,MR119312}. Notice also that $\E\{s(Y)|Y^{(2)}\}=0$ almost surely for $Y=(Y^{(1)},Y^{(2)})\sim F_0$. 
%\item
%\end{enumerate}
\end{example}

\begin{example}\label{ex:gof}
  Let the density $f_2$ %(x^{(2)})$
  be symmetric around $0$, and consider two univariate functions
  $h_1$ and $h_2$ that are both non-constant and bounded by
  $1$ in magnitude, with
  $h_2$ additionally being an odd function.  Let
  $f_1$ be a density such that $\int f_1(x^{(1)})h_1(x^{(1)}) \d
  x^{(1)}\ne 0$.  Then the bivariate density
  $g$ can be chosen such that $s(x)=h_1(x^{(1)})h_2(x^{(2)})$ and
  % , where
  % $h_1(x^{(1)})$ and
  % $h_2(x^{(2)})$ are both non-constant and bounded by $1$, and
  % $h_2$ is an odd function. % satisfies $h_2(x^{(2)})=-h_2(-x^{(2)})$.
  % Then,
  {then Assumption~\ref{asp:far}} holds. %as long as $\int_{-\infty}^{\infty} \{F_{2}(x^{(2)})^2-1/3\}
  % f_2(x^{(2)})h_2(x^{(2)}) \d
  % x^{(2)}\ne0$. %(e.g. $f_2(x^{(2)})=(2\pi)^{-1/2}\exp\{-(x^{(2)})^2/2\}$, $h_2(x^{(2)})=\sin(x^{(2)})$)
  % and $\int f_1(x^{(1)})h_1(x^{(1)}) \d x^{(1)}\ne
  % 0$. %(e.g. $f_1(x^{(1)}),h_1(x^{(1)})>0$),
  {For example, we can take 
  $f_1(t)=f_2(t)=1/2\times\ind(-1\le t\le 1)$, 
  $h_1(t)=|1-2\Psi(t)|$, and $h_2(t)=1-2\Psi(t)$, 
  where $\Psi$ denotes the distribution function 
  of the uniform distribution on $[-1,1]$. 
  In this case, $\E\{s(Y)|Y^{(2)}\}$ is not almost surely
  zero for $Y=(Y^{(1)},Y^{(2)})\sim F_0$.} 
\end{example}

%More examples can be found in \citet[Sec.~3]{MR3466185}. %; see also Example~\ref{ex:far} ahead.

\vspace{0.25cm}

%\begin{remark}
%The proofs for Theorems \ref{thm:powerless-far}--\ref{thm:opt-far} can be adopted to analyse {\it Fr\'echet alternatives} \citep{MR49518}, which also can be treated as special cases of { Dhar--Dassios--Bergsma alternatives} \citep[Sec.~3]{MR3466185}, and similar powerlessness/powerfulness conclusions hold for the considered four tests. We omit the details here due to their intrinsic similarity. 
%while another sets of technical conditions % required, and is hence preferred to be omitted for reading easiness.
%\end{remark}

%\subsection{Sinusoid alternatives}
%
%\[Y = \Delta\cos 8\pi X+ (1-\Delta)\epsilon\]
%where $X\sim {\rm Unif}[-1,1]$ and $\epsilon\sim N(0,1)$ are independent. 
%
%We have
%\begin{align*}
%f_{\mX}(\mx;0) &= \frac12\ind(x\in[-1,1])\cdot \Phi'(y) \\
%f_{\mX}(\mx;\Delta) &= \frac12\ind(x\in[-1,1])\cdot \Phi'\bigg(\frac{y-\Delta\cos 8\pi x}{1-\Delta}\bigg)\cdot\frac{1}{1-\Delta} \\
%L(\mx;\Delta) &=\frac{1}{1-\Delta}\Phi'\bigg(\frac{y-\Delta\cos 8\pi x}{1-\Delta}\bigg)\bigg/\Phi'(y)\\
%L'(\mx;\Delta) &=\bigg[\frac{1}{(1-\Delta)^2}\Phi'\bigg(\frac{y-\Delta\cos 8\pi x}{1-\Delta}\bigg)\\
%&\quad+\frac{1}{1-\Delta}\Phi''\bigg(\frac{y-\Delta\cos 8\pi x}{1-\Delta}\bigg)\frac{-\cos 8\pi x(1-\Delta)+(y-\Delta\cos 8\pi x)}{(1-\Delta)^2}\bigg]\bigg/\Phi'(y)\\
%L'(\mx;0) &=[\Phi'(y)+\Phi''(y)(y-\cos 8\pi x)]/\Phi'(y)=1-y(y-\cos 8\pi x)
%\end{align*}

For a local power analysis in any one of the {two} considered
alternative families, we examine the asymptotic power along a respective sequence of
alternatives obtained as
\begin{align}\label{eq:local-alternative}
H_{1,n}(\Delta_0):\Delta=\Delta_n,  \text{ where } \Delta_n\equiv n^{-1/2}\Delta_0
\end{align}
with some constant $\Delta_0> 0$.
% Testing the null
% hypothesis then reduces to testing 
% \[
% H_0: \Delta_0 = 0~~~\text{versus}~~~H_1: \Delta_0\ne 0.
% \]
% For any such sequence, the joint distribution of an $n$-sample
% determined by $\Delta_n$
% % in the family (A), (B), and (C)
% will be denoted by $\pr_{\Delta_n}.$
% $P_{\Delta_n}^{(A)}$, $P_{\Delta_n}^{(B)}$, and $P_{\Delta_n}^{(C)}$, respectively.  
We obtain the following results on %the power of
the discussed tests.

\begin{theorem}[Power analysis]\label{thm:power}
  Suppose the considered sequences of local alternatives are formed
  such that Assumption~\ref{asp:kon} {or} \ref{asp:far}
  holds when considering a family of type (A) {or} (B), respectively.  Then concerning any sequence  of alternatives given in \eqref{eq:local-alternative}, %with $\Delta_0= C_0$, %as $n\to\infty$:
% Concerning any one of the three local alternative families (A)--(C), as long as the corresponding assumption (Assumption~\ref{asp:kon}, \ref{asp:far}, or \ref{asp:gof}) holds, we have, 
\begin{enumerate}[label=(\roman*)]
\item for any one of the {two} types of alternatives (A) {or} (B), and
  any fixed constant $\Delta_0>0$,
\[
\lim_{n\to\infty}\pr\{{T}^{\xi_n}_{\alpha}=1\mid H_{1,n}(\Delta_0)\}=\alpha;
\yestag\label{eq:subopt}
\]
%where the probability is computed for the sequence of alternatives given by $\Delta_n$ with $\Delta_0= C_0$;
%\item for any one of local alternative families (A) and (B), and any fixed constant $\Delta_0>0$,
%\[
%\lim_{n\to\infty}\pr\{{T}^{\xi_n^*}_{\alpha}=1\mid H_{1,n}(\Delta_0)\}=\alpha;
%\yestag\label{eq:subopt-dette}
%\]
%%where the probability is computed for the sequence of alternatives given by $\Delta_n$ with $\Delta_0= C_0$;
%\item for local alternative family (C) and any number $\beta>0$, there exists some sufficiently large constant 
%$C_\beta>0$ only depending on $\beta$%{, but not, in specific, on $K,h_1,h_2$ in defining $\xi_n^*$ if all assumptions stated in Definition~\ref{def:dette} hold,}
%such that, as long as $\Delta_0>C_\beta$,
%\[
%\lim_{n\to\infty}\pr\{{T}^{\xi_n^*}_{\alpha}=1\mid H_{1,n}(\Delta_0)\}\ge1-\beta;
%\yestag\label{eq:opt-dette}
%\]
%{recall here that the tuning parameters $h_1,h_2$ in $\xi_n^*$ have been assumed to scale as in \eqref{eq:tuning};}
%%where the probability is computed for the sequence of alternatives given by $\Delta_n$ with any fixed $\Delta_0\ge C_\beta$; 
\item for any local alternative family and any number $\beta>0$, 
there exists some sufficiently large constant 
$C_\beta>0$ only depending on $\beta$ such that, as long as $\Delta_0>C_\beta$,
\[
\lim_{n\to\infty}\pr\{{T}^{\mu_n}_{\alpha}=1\mid H_{1,n}(\Delta_0)\}\ge1-\beta,
\yestag\label{eq:opt}
\]
where $\mu_n\in\{D_n, R_n, \tau^*_n\}$. %, 
%and the probability is computed for the sequence of alternatives given by $\Delta_n$ with any fixed $\Delta_0\ge C_\beta$. 
\end{enumerate}
\end{theorem}

{In contrast to Theorem \ref{thm:power}, Proposition~\ref{thm:opt}
  below shows that the power of any size-$\alpha$ test can be
  arbitrarily close to $\alpha$ when $\Delta_0$ is sufficiently small
  in the local alternative model $H_{1,n}(\Delta_0)$. This result
  combined with \eqref{eq:subopt} and \eqref{eq:opt} manifests that %, and Claim~\eqref{eq:opt} shows that 
the size-$\alpha$ tests based on one of $D_n,R_n,\tau^*_n$ %can detect the local alternative $H_{1,n}(\Delta_0)$ from the null with power arbitrarily close to $1$ when $\Delta_0$ is sufficiently large. Hence we call these tests 
are rate-optimal against the considered local alternatives, %. In contrast, Claim~\eqref{eq:subopt} shows that 
while the size-$\alpha$ test based on Chatterjee's correlation coefficient, with only trivial power against the local alternative model $H_{1,n}(\Delta_0)$ for any fixed $\Delta_0$, %, and thus we call the test based on Chatterjee's correlation coefficient 
is rate sub-optimal. }

\begin{proposition}[Rate-optimality]\label{thm:opt}
Concerning any one of the {two} local alternative families and any
sequence  of alternatives given in \eqref{eq:local-alternative}, as
long as the corresponding %assumption 
Assumption~\ref{asp:kon} {or} \ref{asp:far} holds, we have that  for any number $\beta>0$ satisfying $\alpha+\beta<1$ there exists a constant $c_\beta>0$ only depending on $\beta$ such that
\[
\inf_{\overline{{T}}_{\alpha}\in\cT_{\alpha}}\pr\{\overline{{T}}_{\alpha}=0\mid H_{1,n}(c_\beta)\}\geq 1-\alpha-\beta
\]
for all sufficiently large $n$. Here the infimum is
taken over all size-$\alpha$
tests. %, and the supremum is taken over all distributions $\mX$ with $\Delta_n$ such that $|\Delta_0|\ge c_\beta$.
\end{proposition}

{
  \begin{remark}\label{rmk:explain}
    Assumptions~\ref{asp:kon} and \ref{asp:far} are
technical conditions imposed to ensure that (i) the two considered
sequences of alternatives are all locally asymptotically normal
\citep[Chapter~7]{MR1652247}, i.e., the log likelihood ratio processes
admit a quadratic expansion; (ii) the conditional expectation of the
score function given the first margin is almost surely zero. Here the
second requirement was invoked to allow for a use of the conditional
multiplier central limit theorem (cf. Chapter~2.9 in \citealp{MR1385671})
that appears to be the key in
analysing the power of Chatterjee's correlation coefficient.  In
addition to their generality, we would like to emphasize that these
technical assumptions are indeed satisfied by important models such as
Gaussian rotation and Farlie alternatives, which are commonly used to
investigate local power of independence tests.
    \end{remark}
    }

% {
% \begin{remark}\label{rmk:explain}
% Assumptions~\ref{asp:kon}, \ref{asp:far}, and \ref{asp:gof} are technical conditions posed to ensure that (i) the three considered sequences of alternatives are all locally asymptotically normal \citep[Chapter~7]{MR1652247}, i.e., the log likelihood ratio processes admit a quadratic expansion; (ii) conditional expectation of the score function given the first margin is almost surely zero. Here the second requirement was invoked to allow for a use of the conditional multiplier central limit theorem that appears to be the key in analysing the power of Chatterjee's correlation coefficient. In addition to their generality, we would love to highlight that these technical assumptions are satisfied by important independence-testing models like Gaussian rotation and Farlie alternatives, where local powers are commonly compared on.
% \end{remark}
% }

{
\begin{remark}\label{rmk:morealter}
We note that the linear, step function, W-shaped,
sinusoid, and circular alternatives considered in
\citet[Section~4.3]{chatterjee2020new} can all be viewed as generalized rotation
alternatives. The proof techniques used in this paper are hence
directly applicable to these five alternatives by means of a re-parametrization. 
To illustrate this point, consider, for example, the following alternative motivated by \citet[Section~4.3]{chatterjee2020new}:
%\begin{equation}\label{eq:newmodel}
%X^{(2)} = \Delta g(X^{(1)})+ \epsilon,
%\end{equation}
%where $X^{(1)}\sim {\rm Unif}[-1,1]$ and $\epsilon\sim N(0,1)$ are independent, 
%and $g$ is a non-constant measurable function such that $\E[\{g(X^{(1)})\}^2] <\infty$. 
\begin{equation}\label{eq:chattmodel}
X^{(1)} = Y^{(1)}~~~\text{and}~~~
X^{(2)} = \Delta g(Y^{(1)})+ Y^{(2)},
\end{equation}
where $Y^{(1)}$ and $Y^{(2)}$ are independent and absolutely
continuous with respective densities $f_1,f_2$. {Notice that model \eqref{eq:chattmodel} and the one used in \citet[Section~4.3]{chatterjee2020new} are equivalent for rank-based tests as ranks are scale invariant.} Assume then that
\begin{enumerate}[label=(\roman*)]
\item the distributions of $\mX=(X^{(1)},X^{(2)})$ have a common support for all $\Delta\in(-1,1)$;
\item the density $f_2$ is absolutely continuous with non-constant logarithmic derivative $\rho_2\equiv f'_2/f_2$ with $0<\E[\{\rho_2(Y^{(2)})\}^2] <\infty$; 
\item the function $g$ is non-constant and measurable such that $0<\E[\{g(Y^{(1)})\}^2] <\infty$. 
\end{enumerate}
%We have
%\begin{align*}
%f_{ X}( x;0)      &= \frac12\ind(x^{(1)}\in[-1,1])\cdot \varphi(x^{(2)}), \\
%f_{ X}( x;\Delta) &= \frac12\ind(x^{(1)}\in[-1,1])\cdot \varphi(x^{(2)}-\Delta\cos 8\pi x^{(1)}), \\
%L ( x;\Delta) &=\varphi(x^{(2)}-\Delta\cos 8\pi x^{(1)})/\varphi(x^{(2)}), \\
%L'( x;\Delta) &=\varphi'(x^{(2)}-\Delta\cos 8\pi x^{(1)})(-\cos 8\pi x^{(1)})/\varphi(x^{(2)}), \\
%L'( x;0) &=x^{(2)}(\cos 8\pi x^{(1)}). 
%\end{align*}
%and examine the asymptotic power along a respective sequence of alternatives obtained as
%\[
%H_{1,n}(\Delta_0):\Delta=\Delta_n,  \text{ where } \Delta_n\equiv n^{-1/2}\Delta_0
%\]
%with some constant $\Delta_0> 0$. 
%The proof techniques for family~(A) can be adopted to analyse the alternatives \eqref{eq:newmodel}, and similar powerlessness/powerfulness conclusions of the considered five tests for family (A) stated in Theorem~\ref{thm:power} hold for the alternatives \eqref{eq:newmodel}. 
Claims~\eqref{eq:subopt} and \eqref{eq:opt} will then hold for the
alternatives \eqref{eq:chattmodel} in observation of arguments similar to
those made in the proof of Theorem~\ref{thm:power} for the rotation alternatives (A). 
%Claim~\eqref{eq:subopt-dette} holds for the alternatives \eqref{eq:newmodel} if $\E\{g(Y^{(1)})\}=0$; 
%Claim~\eqref{eq:opt-dette} holds for the alternatives \eqref{eq:newmodel} if $\E\{g(Y^{(1)})\}\ne0$ and $\E[\{F_{2}(Y^{(2)})\}^2\rho_2(Y^{(2)})]\ne 0$. 
%The claim about Chatterjee's correlation coefficient is also recorded in \citet[Remark~4.15]{cao2020correlations}. 
\end{remark}
}

%\begin{remark}\label{rmk:CB1}
%  In
%  particular, Assumption II on Page 24 of
%  \citet{cao2020correlations} appears to rule out some
%  interesting %and sufficiently general
%  % local
%  alternatives.  For instance, the mixture alternatives in our
%  Class II do not satisfy Assumption II (see, also, their Remark 4.19). %, and their Lemma~4.13 no
% % longer holds in this case.  
% In this class
%  we found the test using Chatterjee's $\xi_n$ to be rate
%  sub-optimal, recalling \eqref{eq:subopt}, while the one given by
%  Dette--Siburg--Stoimenov's $\xi_n^*$ is rate optimal, recalling
%  \eqref{eq:opt-dette}.
%\end{remark}

\begin{remark}\label{rmk:CB2}
  \citet[Section~4.4]{cao2020correlations} performed {a} local power
  analysis for Chatterjee's $\xi_n$ %and Dette--Siburg--Stoimenov's $\xi_n^*$ 
  under a set of assumptions that differs from ours.  
  The goal of our local power analysis was to exhibit explicitly the{,} 
  at times surprising{,} differences in power of the independence tests
  given by the {four} rank correlation coefficients from
  Definitions~\ref{eg:cha}{, \ref{eg:hd}}--\ref{eg:bdy}.  To this end, we focused on
  rotation and mixture 
  alternatives from the literature.  However, from the proof techniques in
  Section~\ref{subsec:proof-power} {of the supplementary material}, it is evident that
  Claims~\eqref{eq:subopt} and \eqref{eq:opt} hold for further types
  of local alternative families.  For the former claim, which concerns
  lack of power of
  Chatterjee's $\xi_n$, this point has been pursued in Section 4.4 of
  % ; the former observation has been
  % discussed in Section 4.4. (cf. Remark 4.15 therein) of
  \citet{cao2020correlations}.
  % In this paper, however, we only discuss
  % in depth the two types (rotation and mixture) of classical local
  % alternative families in order to present the rate sub-optimality as
  % well as rate optimality of the five rank correlation coefficients
  % defined in Definitions~\ref{eg:cha}--\ref{eg:bdy} explicitly.
\end{remark}

% \begin{remark}
% In view of the proofs in Section~\ref{subsec:proof-power}, it is evident that Claims~\eqref{eq:subopt} and \eqref{eq:opt} hold for more local alternative families; the former observation has been discussed in Section 4.4. (cf. Remark 4.15 therein) of \citet{cao2020correlations}. In this paper, however, we only discuss in depth the two types (rotation and mixture) of  classical local alternative families in order to present the rate sub-optimality as well as rate optimality of the five rank correlation coefficients defined in Definitions~\ref{eg:cha}--\ref{eg:bdy} explicitly. 
% \end{remark}

\section{Rank correlations for discontinuous distributions}\label{sec:newtie}

In this section, we drop the continuity assumption of $F$ made in Sections~\ref{sec:prelim}--\ref{sec:subopt}, and allow for ties to exist with a nonzero probability. 
Among the five correlation coefficients, $\xi_n^*$ is no longer an
appropriate estimator when $F$ is not continuous. 
We will only discuss the properties of the other four estimators $\xi_n$, $D_n$, $R_n$, and $\tau_n^*$.

Recall that the computation issue has been address in Remark \ref{remark:tie-computation}. Our first result in this section focuses on approximation consistency of the correlation coefficients $\xi_n$, $D_n$, $R_n$ and $\tau^*_n$ to their population quantities. % are still  consistent regardless of the continuity of the  distribution. 
To this end, we define 
the families of distribution more general than the ones considered so
far as follows:  
\begin{align*}
\cF&\equiv\bigl\{F: F\text{ is a bivariate distribution function}\bigr\},\\
\cF^{*}&\equiv\bigl\{F: F_k\text{ is not degenerate, i.e., $F_k(x)\ne I(x \ge x_0)$ for any real number $x_0$ for $k=1,2$}\bigr\},\\
\cF^{\tau^*}&\equiv\bigl\{F: F\text{ is discrete, continuous, or a mixture of}\\ 
&\qquad\qquad\text{discrete and jointly absolutely continuous distribution functions}\bigr\}.\yestag\label{eq:bifamily2}
\end{align*}

For the estimators $\xi_n$, $D_n$, $R_n$, and $\tau_n^*$, the following result on consistency can be given.

\begin{proposition}[Consistency of estimators]
%Proposition~\ref{prop:strong}(i) still holds if $X^{(2)}$ is not almost surely a constant (Theorem~1.1 in \citealp{chatterjee2020new}). 
%Proposition~\ref{prop:strong}(iii) still hold for all bivariate distributions
%(Proposition~1 in \citealp{MR3842884} and Theorem~5.4.A in \citealp{MR595165}) and no continuity assumption is required at all. 
As $n\to\infty$, we have
\begin{enumerate}[label=(\roman*)]
\item for $F\in\cF^{*}$, $\xi_n$ converges in probability to $\xi$ (Theorem~1.1 in \citealp{chatterjee2020new}); 
\item for $F\in\cF$, $\mu_n$ converges in probability to $\mu$ for $\mu\in\{D,R,\tau^*\}$ (Proposition~1 in \citealp{MR3842884}, Theorem~5.4.A in \citealp{MR595165}). 
\end{enumerate}
%It also holds that $\E \mu_n =\mu$ for $\mu\in\{D,R,\tau^*\}$ and $n\ge6$ (Proposition~1 in \citealp{MR3842884}, Section~5.1.1 in \citealp{MR595165}). 
\end{proposition}

The following proposition is a generalization of Proposition~\ref{prop:d-consistent}.

\begin{proposition}[Consistency of correlation measures]\label{prop:tie} 
%The following are true:
%\begin{enumerate}[label=(\roman*)]
%\item Proposition~\ref{prop:d-consistent}(i) still holds for $\cF^{\mc}$ replaced by $\cF^{*}$ (Theorem~1.1 in \citealp{chatterjee2020new}); 
%\item Proposition~\ref{prop:d-consistent}(ii)--(iii) still hold for $\cF^{\mc}$ replaced by $\cF$ (Theorem~3.1 in \citealp{MR0029139}, Proposition~3 in \citealp{yanagimoto1970measures}, page 490 of \citealp{MR0125690}); 
%\item Proposition~\ref{prop:d-consistent}(iv) still holds for $\cF^{\mc}$ replaced by $\cF^{\tau^*}$ (Theorem~1 in \citet{MR3178526}, Theorem~6.1 in \citealp{MR4185806}).
%\end{enumerate} 
The following are true:
\begin{enumerate}[label=(\roman*)]
\item for $F\in\cF^{*}$, $\xi\ge 0$ with equality if and only if the pair is independent (Theorem~1.1 in \citealp{chatterjee2020new});
\item for $F\in\cF$, $D\ge 0$; for $F\in\cFD$, $D=0$ if and only if the pair 
is independent (Theorem~3.1 in \citealp{MR0029139}, Proposition~3 in \citealp{yanagimoto1970measures});  
\item for $F\in\cF$, $R\ge 0$ with equality if and only if the pair is independent (page 490 of \citealp{MR0125690});
\item for $F\in\cF^{\tau^*}$, $\tau^*\ge0$ where equality holds
if and only if the variables are independent (Theorem~1 in \citealp{MR3178526}, Theorem~6.1 in \citealp{MR4185806}).
\end{enumerate}
\end{proposition}

The asymptotic distribution theory from
Section~\ref{sec:tests} can also be extended.  
%\begin{remark}[Limiting null distributions]
As the continuity requirement is dropped, the central limit theorems for
Chatterjee's $\xi_n$ still holds.  However, the asymptotic variance
now has a more complicated form and is not necessarily constant across
the null hypothesis of independence (Theorem~2.2 in
\citealp{chatterjee2020new}).  A similar phenomenon arises for the
limiting null distributions of $D_n$, $R_n$ and $\tau_n^*$ when one or
two marginals are not continuous; see Theorem~4.5 and Corollary~4.1 in
\citet{MR3541972} for further discussion.  As a result, permutation
analysis, which is unfortunately computationally much more intensive, is typically invoked to implement a test outside the realm of
continuous distributions.
%\end{remark}

{
\section{Simulation results}
\label{sec:sim}

In order to further examine the power of the tests, we simulate data as  a sample comprised of $n$ independent copies of
$(X^{(1)},X^{(2)})$, for which we consider a suite of different
specifications based on mixture, rotation, and generalized rotation alternatives.

\begin{example}\label{ex:sim-power1}
  For the distribution of $(X^{(1)},X^{(2)})$ we choose the six
  alternatives.  In their specification, $Y^{(1)}$ and $Y^{(2)}$ are
  always independent random variables and
  $\Delta\equiv n^{-1/2}\Delta_0$.
\begin{enumerate}[label=(\alph*)]
\item\label{ex:sim-power1a}  The pair $(X^{(1)},X^{(2)})$ is given by the
  rotation alternative \eqref{eq:rotamodel}, where $Y^{(1)},Y^{(2)}$
  are both standard Gaussian and $\Delta_0=2$.  This is an instance of
  our Example~\ref{ex:kon}.
\item\label{ex:sim-power1b}  The pair $(X^{(1)},X^{(2)})$ is given by the
  mixture alternative \eqref{eq:gof-model}, where
\begin{align*}
F_0\big(x^{(1)},x^{(2)}\big)&\equiv \Psi\big(x^{(1)}\big)\Psi\big(x^{(2)}\big),\\
G\big(x^{(1)},x^{(2)}\big)&\equiv \Psi\big(x^{(1)}\big)\Psi\big(x^{(2)}\big)\big[1+\big\{1-\Psi\big(x^{(1)}\big)\big\}\big\{1-\Psi\big(x^{(2)}\big)\big\}\big],
\end{align*}
$\Psi(\cdot)$ denotes the distribution function of the uniform
distribution on $[-1,1]$, and $\Delta_0=10$.  This is in accordance
with our Example~\ref{ex:far}.
\item\label{ex:sim-power1c}  The pair $(X^{(1)},X^{(2)})$ is given by the
  mixture alternative \eqref{eq:gof-model}, where the density functions of $F$ and $G$, denoted by $f_0$ and $g$, are given by
\begin{align*}
f_0\big(x^{(1)},x^{(2)}\big)&\equiv \psi\big(x^{(1)}\big)\psi\big(x^{(2)}\big),\\
g\big(x^{(1)},x^{(2)}\big)&\equiv \psi\big(x^{(1)}\big)\psi\big(x^{(2)}\big)\big[1+\big|1-2\Psi\big(x^{(1)}\big)\big|\big\{1-2\Psi\big(x^{(2)}\big)\big\}\big],
\end{align*}
$\psi(t)\equiv1/2\times\ind(-1\le t\le 1)$, and $\Delta_0=20$.  This
is an instance of our Example~\ref{ex:gof}.
\item\label{ex:sim-power1d} The pair $(X^{(1)},X^{(2)})$ is given by the
  generalized rotation alternative \eqref{eq:chattmodel}, where $Y^{(1)}$ is uniformly distributed on $[-1,1]$, $Y^{(2)}$ is standard Gaussian, $g$ takes values $-3$, $2$, $-4$, and $-3$ in the intervals $[-1, -0.5)$, $[-0.5, 0)$, $[0, 0.5)$, and $[0.5, 1]$, respectively, and $\Delta_0=3$.
\item\label{ex:sim-power1e} The pair $(X^{(1)},X^{(2)})$ is given by \eqref{eq:chattmodel}, where $Y^{(1)}$ is uniformly distributed on $[-1,1]$, $Y^{(2)}$ is standard Gaussian, $g(t)\equiv |t+0.5|\ind(t<0)+|t-0.5|\ind(t\ge0)$, and $\Delta_0=60$.
\item\label{ex:sim-power1f} The pair $(X^{(1)},X^{(2)})$ is given by \eqref{eq:chattmodel}, where $Y^{(1)}$ is uniformly distributed on $[-1,1]$, $Y^{(2)}$ is standard Gaussian, $g(t)\equiv \cos(2\pi t)$, and $\Delta_0=12$. 
\end{enumerate}
\end{example}

As indicated, the first three simulation settings are taken from
Examples~\ref{ex:kon}--\ref{ex:gof}.  The latter three are motivated by 
step function, W-shaped, and 
sinusoid settings in which Chatterjee's 
correlation coefficient performs well; see 
\citet[Section~4.3]{chatterjee2020new}.

Our focus is on comparing the empirical performance of the five tests
${T}^{\xi_n}_{\alpha}$, ${T}^{\xi_n^*}_{\alpha}$,
${T}^{D_n}_{\alpha}$, ${T}^{R_n}_{\alpha}$,
${T}^{\tau^*_n}_{\alpha}$.  The first four tests are conducted using the
asymptotics from Proposition \ref{prop:cha}.  The last test is
implemented {with bandwidths chosen as $h_1=h_2=n^{-3/10}$ following the suggestion in Section~6.1 of \citet{MR3024030} and using a finite-sample critical value}, which we approximate
via 1000 Monte Carlo simulations. The nominal significance level
is set to $0.05$, and the sample size is chosen as
$n\in\{500,1000,5000,10000\}$.  For each of the six settings and four sample
sizes, we conduct $1000$ simulations.

Before turning to statistical properties, we contrast the computation 
times for calculating the five considered rank correlation coefficients first.  
Table~\ref{tab:time} shows times in the considered rotation setting (a); 
the results for other settings are essentially the same. 
The calculations of $\xi_n$ and $\xi_n^*$ are by our own implementation, 
and those of $D_n$, $R_n$, $\tau_n^*$ are made using the functions 
\texttt{.calc.hoeffding()}, \texttt{.calc.refined()}, 
and \texttt{.calc.taustar()} from R package \texttt{independence} 
\citep{R:independence}, respectively.  All experiments are conducted on a laptop with a 2.6
GHz Intel Core i5 processor and a 8 GB memory.
One observes the clear computational advantages
of $\xi_n$, $D_n$, $R_n$, and $ \tau_n^*$ over \citet{MR3024030}'s
estimator $\xi_n^*$.  
% Chatterjee's new correlation coefficient and three well-known rank
% correlation coefficients, namely, Hoeffding's $D$,
% Blum--Kiefer--Rosenblatt's $R$, and Bergsma--Dassios--Yanagimoto's
% $\tau^*$, over \citet{MR3024030}'s estimator.
The difference in computation time between Chatterjee's coefficient $\xi_n$ and
Hoeffding's $D_n$ is insignificant.   Both $\xi_n$ and $D_n$ are slightly
faster to compute than Blum--Kiefer--Rosenblatt's $R_n$ and
Bergsma--Dassios--Yanagimoto's $\tau^*_n$; computation times differ by
a factor less than 2.5.

Table \ref{tab:power1} shows the empirical powers %(rejection frequencies) 
of the five tests. %, based on~the $1000$ simulations
%conducted in each case.  
The results confirm our earlier theoretical
claims on the powers of the  different tests in the different models, 
that Hoeffding's $D$, Blum--Kiefer--Rosenblatt's $R$, and Bergsma--Dassios--Yanagimoto's $\tau^*$ outperform Chatterjee's correlation coefficient in all the settings considered. %To implement the test based on Dette--Siburg--Stoimenov's $\xi_n^*$, we use an exact critical value %for rejection of $H_0$ 
%approximated by Monte Carlo simulation
%since this estimator is rank-based and distribution-free.
Interestingly, the simulation results suggest that the test based on $\xi_n^*$ may have
non-trivial power against certain alternatives; 
see results for Example~\ref{ex:sim-power1}\ref{ex:sim-power1e},\ref{ex:sim-power1f} 
in Table~\ref{tab:power1}.
}

{
\renewcommand{\tabcolsep}{1.5pt}
\renewcommand{\arraystretch}{1.0}
\begin{table}[!htb]
\centering
\caption{\nb{A comparison of computation time for all the five correlation statistics. 
The computation time here is the total time in seconds of 1000 replicates.}}\label{tab:time}{\nb{
\begin{tabular}{cC{.75in}C{.75in}C{.75in}C{.75in}C{.75in}}
\toprule
  $n$  &  $\xi_n$  & $\xi_n^*$  &   $D_n$   &   $R_n$   & $\tau_n^*$ \\
\midrule
     500  &   0.157   &    12.57   &   0.158   &   0.263   &   0.253        \\
    1000  &   0.239   &    33.75   &   0.267   &   0.505   &   0.468        \\
    5000  &   1.655   &    401.4   &   1.823   &   3.601   &   3.087        \\
   10000  &   3.089   &   1152.6   &   3.315   &   7.607   &   7.132        \\
\bottomrule
\end{tabular}}}
%Results are averaged over $5000$ simulated data sets.
\end{table} 
}

{
\renewcommand{\tabcolsep}{1.5pt}
\renewcommand{\arraystretch}{1.0}
\begin{table}[!htb]
\centering
\caption{\nb{Empirical powers of the five competing tests in Example \ref{ex:sim-power1}. The empirical powers here are based on 1000 replicates.}}{\nb{
\begin{tabular}{cC{.45in}C{.45in}C{.45in}C{.45in}C{.45in}
         C{.15in}C{.45in}C{.45in}C{.45in}C{.45in}C{.45in}}
\toprule
  $n$  &  $\xi_n$  & $\xi_n^*$ &   $D_n$   &   $R_n$   & $\tau_n^*$ 
     & &  $\xi_n$  & $\xi_n^*$ &   $D_n$   &   $R_n$   & $\tau_n^*$ \\
\midrule
       &  \multicolumn{5}{c}{Results for Example \ref{ex:sim-power1}\ref{ex:sim-power1a}}  
     & &  \multicolumn{5}{c}{Results for Example \ref{ex:sim-power1}\ref{ex:sim-power1d}}  \\
  500  &   0.103   &   0.178   &   0.954   &   0.955   &   0.957        & &   0.443   &   0.122   &   0.913   &   0.921   &   0.919        \\
 1000  &   0.067   &   0.106   &   0.956   &   0.956   &   0.956        & &   0.285   &   0.111   &   0.923   &   0.928   &   0.927        \\
 5000  &   0.043   &   0.078   &   0.953   &   0.952   &   0.952        & &   0.081   &   0.083   &   0.936   &   0.936   &   0.937        \\
10000  &   0.045   &   0.058   &   0.951   &   0.952   &   0.952        & &   0.081   &   0.052   &   0.955   &   0.954   &   0.955        \\
\vspace{-.5em}\\
       &  \multicolumn{5}{c}{Results for Example \ref{ex:sim-power1}\ref{ex:sim-power1b}}  
     & &  \multicolumn{5}{c}{Results for Example \ref{ex:sim-power1}\ref{ex:sim-power1e}}  \\
  500  &   0.087   &   0.138   &   0.898   &   0.896   &   0.897        & &   0.719   &   1.000   &   0.654   &   0.635   &   0.643        \\
 1000  &   0.067   &   0.089   &   0.900   &   0.900   &   0.899        & &   0.486   &   1.000   &   0.700   &   0.682   &   0.692        \\
 5000  &   0.059   &   0.082   &   0.891   &   0.890   &   0.891        & &   0.146   &   1.000   &   0.735   &   0.735   &   0.736        \\
10000  &   0.052   &   0.045   &   0.911   &   0.914   &   0.915        & &   0.105   &   0.997   &   0.754   &   0.752   &   0.752        \\
\vspace{-.5em}\\
       &  \multicolumn{5}{c}{Results for Example \ref{ex:sim-power1}\ref{ex:sim-power1c}}  
     & &  \multicolumn{5}{c}{Results for Example \ref{ex:sim-power1}\ref{ex:sim-power1f}}  \\
  500  &   0.088   &   0.559   &   0.412   &   0.404   &   0.410        & &   0.688   &   1.000   &   0.635   &   0.603   &   0.611        \\
 1000  &   0.066   &   0.408   &   0.390   &   0.391   &   0.396        & &   0.459   &   1.000   &   0.669   &   0.655   &   0.660        \\
 5000  &   0.060   &   0.327   &   0.363   &   0.364   &   0.364        & &   0.141   &   1.000   &   0.717   &   0.712   &   0.713        \\
10000  &   0.048   &   0.248   &   0.392   &   0.395   &   0.396        & &   0.100   &   0.994   &   0.726   &   0.730   &   0.728        \\
\bottomrule 
\end{tabular}}}
\label{tab:power1}
%Results are averaged over $5000$ simulated data sets.
\end{table}
}

%}

\section{Discussion} 
\label{sec:dis}

In this paper we considered independence tests based on the five rank
correlations from Definitions~\ref{eg:cha}--\ref{eg:bdy}.  As we
surveyed in Section~\ref{sec:prelim}, recent advances lead to little
difference in the efficiency of known algorithms to compute these
correlation coefficients.  For continuous distributions, i.e., data
without ties, all correlations except for {Dette--Siburg--Stoimenov's} $\xi_n^*$ can be
computed in nearly linear time.  Moreover, all but Hoeffding's $D$
give consistent tests of independence for arbitrary continuous
distributions; consistency of $D$ can be established for all
absolutely continuous distributions.

Our main new contribution is a local power analysis for continuous
distributions that revealed interesting differences in the power of
the tests.  This analysis features subtle differences but the
take-away message is that $\xi_n$ {is} suboptimal for
testing independence, whereas the more classical {$D_n$, $R_n$, and
$\tau^*_n$} are rate optimal in the considered setup.  This said, $\xi_n$
and $\xi_n^*$ have very appealing properties that do not pertain to
independence but rather detection of perfect functional dependence.  We refer
the reader to \citet{MR3024030} and \citet{chatterjee2020new} as well as \citet{cao2020correlations}.

We summarize the properties discussed in our paper in
Table~\ref{tab:propof}.  When referring to independence tests in this
table we assume continuous observations, i.e., $F\in\cF^{c}$.
Moreover, when discussing $\xi_n^*$, we assume additionally that the
kernel $K$ and bandwidths $h_1,h_2$ satisfy all assumptions stated in
Definition~\ref{def:dette}.  The table features two rows for
computation, where the first pertains to continuous observations free
of ties and the second pertains to arbitrary observations. The third 
row of the table concerns consistency of correlation measures; refer to 
\eqref{eq:bifamily1} and \eqref{eq:bifamily2} for the definitions of 
table entries. The fourth row concerns consistency of independence 
tests assuming $F\in\cF^{c}$. Finally, we summarize the rate-optimality 
and rate sub-optimality of five independence tests under {two} 
local alternatives (A) {and} (B) considered in Section~\ref{sec:subopt}.

%\begin{landscape}
{
\renewcommand{\tabcolsep}{4pt}
\renewcommand{\arraystretch}{1}
\begin{table}[!htb]
\begin{center}
\caption{Properties of the five rank correlation coefficients defined in Definitions~\ref{eg:cha}--\ref{eg:bdy}.}\label{tab:propof}
\begin{tabular}{cccC{.8in}C{.8in}C{.8in}C{.8in}C{.8in}}
\toprule
\multicolumn{3}{c}{$\mu_n$}                             
        & $\xi_n$ 
        & $\xi_n^*$ 
        & $D_n$ 
        & $R_n$ 
        & $\tau_n^*$ \\
\midrule
        & \multicolumn{1}{C{0.7in}}{Computa-}
        & $F\in\cF^{c}$
        & $O(n\log n)$
        & $O(n^{5/3})$%{\color{blue}$^{(a)}\!\!\!$}
        & $O(n\log n)$ 
        & $O(n\log n)$ 
        & $O(n\log n)$ \\
%\hline
  (i)   & \multicolumn{1}{C{0.7in}}{tional} \\
        & \multicolumn{1}{C{0.7in}}{efficiency}
        & $F\in\cF$
        & $O(n\log n)$
        & ------ 
        & $O(n\log n)$ 
        & $O(n^2)$ 
        & $O(n^2)$ \\[1em]
%\hline
 (ii)   & \multicolumn{2}{C{1.2in}}{Consistency of correlation measures}        
        & $F\in\cF^{*}${\color{blue}$^{(a)}$} 
        & $F\in\cF^{*}$  
        & $F\in\cFD$  
        & $F\in\cF$   
        & $F\in\cF^{\tau^*}$ \\[1.5em]
%\cline{2-8}
 (ii')  & \multicolumn{2}{C{1.2in}}{Consistency of independence tests}  
        & $F\in\cF^{c}$
        & $F\in\cF^{\rm DSS}$
        & $F\in\cFD$
        & $F\in\cF^{c}$
        & $F\in\cF^{c}$ \\[1em]
%\hline
\multirow{2}{*}{\raisebox{-5mm}{(iii)}}
        & \multicolumn{1}{C{0.7in}}{\raisebox{-3.5mm}{Statistical}}
        & (A)  
        & rate sub-optimal 
        & \nb{------} 
        & rate-optimal 
        & rate-optimal 
        & rate-optimal \\[1em]
%\cline{3-8}
        & \multicolumn{1}{C{0.7in}}{\raisebox{4.5mm}{efficiency}}     
        & (B)  
        & rate sub-optimal 
        & \nb{------} 
        & rate-optimal 
        & rate-optimal 
        & rate-optimal \\
%\cline{2-8}
%        & 
%        & (C)
%        & rate sub-optimal 
%        & \nb{------} 
%        & rate-optimal 
%        & rate-optimal 
%        & rate-optimal \\
\bottomrule
\end{tabular}
\begin{minipage}{6in}
\vspace{0.2cm}
{\small 
%{\color{blue}$^{(a)}$} $\delta$ is any arbitrarily small constant \\
{\color{blue}$^{(a)}$} Recall the definitions of bivariate distribution families in \eqref{eq:bifamily1} and \eqref{eq:bifamily2} \\}
\vspace{-1cm}
\end{minipage}
\end{center}
\end{table}
}
%\end{landscape}

\appendix

\section{Proofs}
\label{sec:proofs}

%\subsection{Proofs for Section \ref{sec:alter-kon}}
Throughout the proofs below, all the claims regarding conditional
expectations, conditional variances, and conditional covariances are in the almost sure sense.

\subsection{Proof of Proposition~\ref{prop:strong} ($\xi^*_n$)}\label{subsec:proof:prop:strong}

\begin{proof}[Proof of Proposition~\ref{prop:strong} ($\xi^*$)]

Equation (21) in \citet{MR3024030} states that
\[\hat B_{2n}-\tilde B_{2n}-C_{1n}-C_{2n}=o_{p}(n^{-1/2}),\]
but tracking a glitch in signs the equation should in fact be \[\hat B_{2n}-\tilde B_{2n}+C_{1n}+C_{2n}=o_{p}(n^{-1/2}).\]  
Accordingly, a revised version of Equations~(24)--(26) in \citet{MR3024030} shows that, 
\[
n^{1/2}(\xi_n^*-\xi)=\frac{12}{n^{1/2}}\sum_{i=1}^{n}(Z_i-\E Z_i)+o_{p}(1)
\yestag\label{eq:detteasym1}
\]
where $Z_i\equiv Z_{i,1}-Z_{i,2}-Z_{i,3}$ with 
\begin{align*}
Z_{i,1}&\equiv \int_{0}^{1}\ind\Big\{F_{Y^{(2)}}\Big(Y^{(2)}_i\Big)\le u^{(2)}\Big\}\tau\Big(F_{Y^{(1)}}\Big(Y^{(1)}_i\Big), u^{(2)}\Big)\d u^{(2)},\\
Z_{i,2}&\equiv \int_{0}^{1}\int_{0}^{1}\ind\Big\{F_{Y^{(1)}}\Big(Y^{(1)}_i\Big)\le u^{(1)}\Big\}\tau\big(u^{(1)}, u^{(2)}\big)\frac{\partial}{\partial u^{(1)}}\tau\big(u^{(1)}, u^{(2)}\big)\d u^{(1)}\d u^{(2)},\\
Z_{i,3}&\equiv \int_{0}^{1}\int_{0}^{1}\ind\Big\{F_{Y^{(2)}}\Big(Y^{(2)}_i\Big)\le u^{(2)}\Big\}\tau\big(u^{(1)}, u^{(2)}\big)\frac{\partial}{\partial u^{(2)}}\tau\big(u^{(1)}, u^{(2)}\big)\d u^{(1)}\d u^{(2)},
\end{align*}
$\tau(u^{(1)}, u^{(2)})={\partial}C(u^{(1)}, u^{(2)})/{\partial u^{(1)}}$, 
and $C(u^{(1)}, u^{(2)})$ is the copula of $(Y^{(1)}, Y^{(2)})$. 
Since the first term on the right hand side of \eqref{eq:detteasym1} has finite variance (see computation on pages 34--35 of \citet{MR3024030}), we deduce that
\[\xi^*_n\longrightarrow_p \xi.
\]
This completes the proof.
\end{proof}

\subsection{Proof of Proposition~\ref{prop:cha}(ii)}\label{subsec:proof:prop:cha}

\begin{proof}[Proof of Proposition~\ref{prop:cha}(ii)]

%Equation (21) in \citet{MR3024030} states that
%\[\hat B_{2n}-\tilde B_{2n}-C_{1n}-C_{2n}=o_{\rm P}(n^{-1/2}),\]
%which should be \[\hat B_{2n}-\tilde B_{2n}+C_{1n}+C_{2n}=o_{\rm P}(n^{-1/2}).\]  
%Accordingly, revised Equations~(24)--(26) in \citet{MR3024030} show that, 
%\[
%n^{1/2}\xi_n^*-\frac{12}{n^{1/2}}\sum_{i=1}^{n}(Z_i-\E Z_i)\to 0 ~~~\text{in probability},
%\yestag\label{eq:detteasym1}
%\]
%where $Z_i\equiv Z_{i,1}-Z_{i,2}-Z_{i,3}$ with 
%\begin{align*}
%Z_{i,1}&\equiv \int_{0}^{1}\ind\Big\{F_{Y^{(2)}}\Big(Y^{(2)}_i\Big)\le u^{(2)}\Big\}\tau\Big(F_{Y^{(1)}}\Big(Y^{(1)}_i\Big), u^{(2)}\Big)\d u^{(2)},\\
%Z_{i,2}&\equiv \int_{0}^{1}\int_{0}^{1}\ind\Big\{F_{Y^{(1)}}\Big(Y^{(1)}_i\Big)\le u^{(1)}\Big\}\tau\big(u^{(1)}, u^{(2)}\big)\frac{\partial}{\partial u^{(1)}}\tau\big(u^{(1)}, u^{(2)}\big)\d u^{(1)}\d u^{(2)},\\
%Z_{i,3}&\equiv \int_{0}^{1}\int_{0}^{1}\ind\Big\{F_{Y^{(2)}}\Big(Y^{(2)}_i\Big)\le u^{(2)}\Big\}\tau\big(u^{(1)}, u^{(2)}\big)\frac{\partial}{\partial u^{(2)}}\tau\big(u^{(1)}, u^{(2)}\big)\d u^{(1)}\d u^{(2)},
%\end{align*}
%$\tau(u^{(1)}, u^{(2)})={\partial}C(u^{(1)}, u^{(2)})/{\partial u^{(1)}}$, 
%and $C(u^{(1)}, u^{(2)})$ is the copula of $(Y^{(1)}, Y^{(2)})$. 
Applying \eqref{eq:detteasym1}, 
it holds under the null that
\[
C(u^{(1)}, u^{(2)})=u^{(1)}u^{(2)}, ~~~\tau(u^{(1)}, u^{(2)})=u^{(2)}. 
\]
 Accordingly,
\[
Z_{i,1}=Z_{i,3}=\int_{0}^{1}\ind\Big\{F_{Y^{(2)}}\Big(Y^{(2)}_i\Big)\le u^{(2)}\Big\}u^{(2)}\d u^{(2)}=\frac{1}{2}\Big[1-\Big\{F_{Y^{(2)}}\Big(Y^{(2)}_i\Big)\Big\}^2\Big]
~~~\text{and}~~~Z_{i,2}=0,
\]
which yields
\[
n^{1/2}\xi_n^*\to 0 ~~~\text{in probability}.
\yestag\label{eq:detteasym2}
\]
This completes the proof.
\end{proof}

\subsection{Proof of Remark~\ref{rmk:kon}}\label{subsec:proof:rmk:kon}

\begin{proof}[Proof of Remark~\ref{rmk:kon}]
Recall that $f_{\mX}(\mx;\Delta)$ denotes the density of $\mX$ with $\Delta$. Denote
\[
L(\mx;\Delta)\equiv \frac{f_{\mX}(\mx;\Delta)}{f_{\mX}(\mx;0)}~~~\text{and}~~~
L'(\mx;\Delta)\equiv \frac{\partial}{\partial\Delta}L(\mx;\Delta).  
\]
These definitions make sense by
Assumption~\ref{asp:kon}\ref{asp:kon1},\ref{asp:kon2}, and we may write  $\cI_{\mX}(0)=\E[\{L'(\mY;0)\}^2]$. Notice that $\mY$ is distributed as $\mX$ with $\Delta=0$. 
Since $\mY=\fA_{\Delta}^{-1}\mX$ is an invertible linear transformation, the density of $\mX$ can be expressed as  
\[
f_{\mX}(\mx;\Delta)=|\det(\fA_{\Delta})|^{-1}f_{\mY}(\fA_{\Delta}^{-1}\mx),
\]
where 
$
f_{\mY}(\my)=f_{\mY}(y^{(1)},y^{(2)})=f_{1}(y^{(1)})f_{2}(y^{(2)}). 
$ 
Direct computation yields
\begin{align*}
&L(\mx;\Delta)=|\det(\fA_{\Delta})|^{-1}f_{\mY}(\fA_{\Delta}^{-1}\mx)\Big/f_{\mY}(\mx),\\
\text{and}~~~
&L'(\mx;0)=-x^{(1)}\Big\{\rho_{2}\Big(x^{(2)}\Big)\Big\}
           -x^{(2)}\Big\{\rho_{1}\Big(x^{(1)}\Big)\Big\}.\yestag\label{eq:whynontrivial}
\end{align*}
Thus $\E\{(Y^{(k)})^2\} <\infty$ and $\E[\{\rho_{k}(Y^{(k)})\}^2] <\infty$ for $k=1,2$ will imply $\cI_{\mX}(0)=\E[\{L'(\mY;0)\}^2]<\infty$ under the Konijn alternatives. Also, $\E[\{\rho_{k}(Y^{(k)})\}^2] <\infty$ implies that $\E\{\rho_{k}(Y^{(k)})\}=0$ by Lemma~A.1 (Part A) in \citet{MR2128239}. %and Lemma~I.2.4.a in \citet{MR0229351}. 
\end{proof}

\subsection{Proof of Example~\ref{ex:kon}}

\begin{proof}[Proof of Example~\ref{ex:kon}]
%\mbox{}
%\begin{enumerate}[label=(\alph*)]
%\item
Assumption~\ref{asp:kon}\ref{asp:kon1} is satisfied since $f_{k}(z)>0$,~$k=1,2$ for all real $z$. 
Assumption~\ref{asp:kon}\ref{asp:kon3} holds in view of
\eqref{eq:whynontrivial}; notice that $L'(\mx;0)$ can never always be $0$. 
For Assumption~\ref{asp:kon}\ref{asp:kon2}, if $\rho_k(z)$ is constant, then $f_{k}(z)$ is either constant or proportional to $e^{Cz}$ with some constant $C$ for all real $z$, which is impossible. Then Assumption~\ref{asp:kon} is satisfied. 

Regarding the special case, 
without loss of generality, we can assume $Y_1$ and $Y_2$ to be standard normal or standard $t$-distributed. For the standard normal, we have $\rho_{k}(z)=-t$ and thus \eqref{eq:kon1} is satisfied. For the standard $t$-distribution with $\nu_k$ degrees of freedom, we have 
$\rho_{k}(z)=-{z(1+1/\nu_k)}/{(1+z^2/\nu_k)}$. It is easy to check \eqref{eq:kon1} is satisfied when $\nu_k>2$.
%\end{enumerate}
%The proof is thus completed.
\end{proof}

\subsection{Proof of Remark~\ref{rmk:far}}\label{subsec:proof:rmk:far}

\begin{proof}[Proof of Remark~\ref{rmk:far}]

Let $f_{\mX}(\mx;\Delta)$ denote the density of $\mX$ with $\Delta$. Denote
\[
L(\mx;\Delta)\equiv \frac{f_{\mX}(\mx;\Delta)}{f_{\mX}(\mx;0)}~~~\text{and}~~~
L'(\mx;\Delta)\equiv \frac{\partial}{\partial\Delta}L(\mx;\Delta), 
\]
then we can write $\cI_{\mX}(0)=\E[\{L'(\mY;0)\}^2]$, where $\mY$ is distributed as $\mX$ with $\Delta=0$. 
Direct computation yields 
\[
L(\mx;\Delta)=\frac{(1-\Delta)f_{0}(\mx)+\Delta g(\mx)}{f_0(\mx)},~~~
L'(\mx;0)=\frac{g(\mx)-f_0(\mx)}{f_0(\mx)},
\]
and thus
\begin{align*}
\cI_{\mX}(0)&=\E[\{L'(\mY;0)\}^2]=\E[\{g(Y)/f_0(Y)-1\}^2]\\
&=\E[\{s(Y)\}^2]=\chi^2(G,F_0)\equiv \int (\d G/\d F_0-1)^2 \d F_0.
\end{align*}
Since $s(x)=g(x)/f_0(x)-1$ is continuous and both $g$ and $f_0$ have compact support, $s(x)$ is bounded. Hence $\cI_{\mX}(0)<\infty$. 
\end{proof}

\subsection{Proof of Example~\ref{ex:far}}

\begin{proof}[Proof of Example~\ref{ex:far}]

To verify Assumption~\ref{asp:far} for the Farlie alternatives, 
we first prove that $G$ is a bonafide joint distribution function.  The corresponding density $g$ is given by
\[
g(x^{(1)},x^{(2)})=f_1(x^{(1)})f_2(x^{(2)})[1+\{1-2F_1(x^{(1)})\}\{1-2F_2(x^{(2)})\}],
\]
which is a bonafide joint density function \citep[Sec.~1.1.5]{MR2288045}. 
Then we have 
\[s(x)=g(x)/f_0(x)-1=\{1-2F_1(x^{(1)})\}\{1-2F_2(x^{(2)})\}\]
and find that
\begin{align*}
&\E[s(Y)|Y^{(1)}]=\{1-2F_1(Y^{(1)})\}\times\E\{1-2F_2(Y^{(2)})\}=0\\
\text{and}~~~
&\E[s(Y)|Y^{(2)}]=\E\{1-2F_1(Y^{(1)})\}\times\{1-2F_2(Y^{(2)})\}=0.
\end{align*}
The proof is completed. 
\end{proof}

%\subsection{Proof of Remark~\ref{rmk:gof}}
%
%\begin{proof}[Proof of Remark~\ref{rmk:gof}]
%If $\E[s(Y)|Y^{(2)}]=0$ almost surely, then we have
%\begin{align*}
%\E\Big[\Big\{F_{2}(Y^{(2)})^2-\frac13\Big\}s(Y)\Big]
%&=\E\bigg(\E\Big[\Big\{F_{2}(Y^{(2)})^2-\frac13\Big\}s(Y)\Big\rvert Y^{(2)}\Big]\bigg)\\
%&=\E\bigg(\Big\{F_{2}(Y^{(2)})^2-\frac13\Big\}\cdot\E\Big[s(Y)\Big\rvert Y^{(2)}\Big]\bigg)=0.
%\end{align*}
%The proof is completed. 
%\end{proof}

\subsection{Proof of Example~\ref{ex:gof}}

\begin{proof}[Proof of Example~\ref{ex:gof}]
%Let $f_{\mX}(\mx;\Delta)$ denote the density of $\mX$ with $\Delta$. Denote
%\[
%L(\mx;\Delta)\equiv \frac{f_{\mX}(\mx;\Delta)}{f_{\mX}(\mx;0)}~~~\text{and}~~~
%L'(\mx;\Delta)\equiv \frac{\partial}{\partial\Delta}L(\mx;\Delta), 
%\]
%then we can write $\cI_{\mX}(0)=\E[\{L'(\mY;0)\}^2]$, where $\mY$ is distributed as $\mX$ with $\Delta=0$. Direct computation yields 
%\[
%L(\mx;\Delta)\equiv \frac{(1-\Delta)f_{0}(\mx)+\Delta g(\mx)}{f_0(\mx)},~~~
%L'(\mx;0)=\frac{g(\mx)-f_0(\mx)}{f_0(\mx)},
%\]
%and thus
%\[
%\cI_{\mX}(0)=\E[\{L'(\mY;0)\}^2]=\E[\{g(Y)/f_0(Y)-1\}^2]=\E[\{s(Y)\}^2]
%=\chi^2(G,F_0)\equiv \int (\d G/\d F_0-1)^2 \d F_0.
%\]
We first verify that $g$ is a bonafide joint density function. First since both $h_1$ and $h_2$ are bounded by $1$, 
\[\lvert g(x)/f_0(x)-1\rvert=\lvert h_1(x^{(1)})h_2(x^{(1)})\rvert\le1,\]
and thus $g(x)\ge0$. Then we write
\[
g(x^{(1)},x^{(2)})=
f_1(x^{(1)})f_2(x^{(2)}) +
f_1(x^{(1)})h_1(x^{(1)})f_2(x^{(2)})h_2(x^{(2)})
\]
and 
\begin{align*}
\int_{-\infty}^{\infty}\int_{-\infty}^{\infty} g(x^{(1)},x^{(2)})\d x^{(1)}\d x^{(2)}
=\;&
 \int_{-\infty}^{\infty}f_1(x^{(1)})\d x^{(1)}\times
 \int_{-\infty}^{\infty}f_2(x^{(2)})\d x^{(2)}\\
&+
 \int_{-\infty}^{\infty}f_1(x^{(1)})h_1(x^{(1)})\d x^{(1)}\times
 \int_{-\infty}^{\infty}f_2(x^{(2)})h_2(x^{(2)})\d x^{(2)}=1,
\end{align*}
where
\[
\int_{-\infty}^{\infty}f_1(x^{(1)})h_1(x^{(1)})\d x^{(1)}<\infty
~~~\text{and}~~~
\int_{-\infty}^{\infty}f_2(x^{(2)})h_2(x^{(2)})\d x^{(2)}=0
\]
since $h_1(x^{(1)}),h_2(x^{(2)})$ are bounded by $1$ and $f_2(x^{(2)})h_2(x^{(2)})=-f_2(-x^{(2)})h_2(-x^{(2)})$. We also have
\begin{align*}
&\E[s(Y)|Y^{(1)}]=h_1(Y^{(1)})\times\E[h_2(Y^{(2)})]=h_1(Y^{(1)})\int_{-\infty}^{\infty}f_2(x^{(2)})h_2(x^{(2)})\d x^{(2)}=0,\\
\text{and}~~~
&\E[s(Y)|Y^{(2)}]=\E[h_1(Y^{(1)})]\times h_2(Y^{(2)})~~~\text{with}~\E[h_1(Y^{(1)})]=\int_{-\infty}^{\infty}f_1(x^{(1)})h_1(x^{(1)})\d x^{(1)}\ne 0.
\end{align*}
The proof is completed. 
\end{proof}

\subsection{Proof of Theorem~\ref{thm:power}(i)}\label{subsec:proof-power}

\begin{proof}[Proof of Theorem~\ref{thm:power}(i)]
%$ $\newline
{\bf (A)} {\it This proof uses all of Assumption~\ref{asp:kon}.}
%This is a direct application of a corollary to Le~Cam's third lemma 
%(\citealp[Sec.~VI.1.3]{MR0229351} and \citealp[Example~6.7]{MR1652247}).
Let $\mY_{i}=(Y^{(1)}_{i},Y^{(2)}_{i})$, $i=1,\dots,n$ be independent copies of $\mY$. %, where $\mY$ is distributed as $\mX$ with $\Delta=0$.  
%Let $F^{(0)}$ and $F^{(a)}$ be the (joint) distribution functions of $(\mY_{1},\dots,\mY_{n})$ and $(\mX_{1},\dots,\mX_{n})$, respectively. % and define $\Lambda_n\equiv \log({\d F^{(a)}}/{\d F^{(0)}})\equiv \log({\d F^{(a)}(\mY_{1},\dots,\mY_{n})}/{\d F^{(0)}(\mY_{1},\dots,\mY_{n})})$.  
Recall that $f_{\mX}(\mx;\Delta)$ is the density of $\mX$ with $\Delta$. 
Denote
\[
L(\mx;\Delta)\equiv \frac{f_{\mX}(\mx;\Delta)}{f_{\mX}(\mx;0)},~~~
L'(\mx;\Delta)\equiv \frac{\partial}{\partial\Delta}L(\mx;\Delta), 
\]
and define $\Lambda_n=\sum_{i=1}^{n}\log L(\mY_i;\Delta_n)$ and $T_n\equiv \Delta_n\sum_{i=1}^{n}L'(\mY_{i};0)$. 
These definitions make sense by Assumption~\ref{asp:kon}\ref{asp:kon1},\ref{asp:kon2}. 

To employ a corollary to Le Cam's third lemma, 
we wish to derive the joint limiting null distribution of $(-n^{1/2}\xi_n/3,\Lambda_n)$.
Under the null hypothesis, it holds that $Y^{(2)}_{[1]},\dots,Y^{(2)}_{[n]}$ are still independent and identically distributed, where $[i]$ is  such that $Y^{(1)}_{[1]}<\cdots<Y^{(1)}_{[n]}$.
%We hence assume hereafter, without loss of generality, that $Y^{(1)}_{1}<\cdots<Y^{(1)}_{n}$. 
In view of \citet[Equation~(9)]{MR1378827}, we have that under the null, 
\[
\Big(-n^{1/2}\xi_n\Big/3\Big)
%=n^{1/2}\sum_{i=1}^{n}\frac{\lvert F^{(n)}_{Y^{(2)}}(Y^{(2)}_{i+1})-F^{(n)}_{Y^{(2)}}(Y^{(2)}_{i})\rvert}{n-1/n}-\frac13 
-n^{-1/2}\sum_{i=1}^{n-1}\Xi_{[i]}\to 0 ~~~\text{in probability},
\yestag\label{eq:Angus}
\]
where %$F^{(n)}_{Y^{(2)}}$ is the empirical cumulative distribution function for $Y^{(2)}_1,\dots,Y^{(2)}_n$,
\begin{align*}
\Xi_{[i]}\equiv \Big\lvert F_{Y^{(2)}}\Big(Y^{(2)}_{[i+1]}\Big)-F_{Y^{(2)}}\Big(Y^{(2)}_{[i]}\Big)\Big\rvert
&+F_{Y^{(2)}}\Big(Y^{(2)}_{[i+1]}\Big)\Big\{1-F_{Y^{(2)}}\Big(Y^{(2)}_{[i+1]}\Big)\Big\}\\
&+F_{Y^{(2)}}\Big(Y^{(2)}_{[i]}\Big)\Big\{1-F_{Y^{(2)}}\Big(Y^{(2)}_{[i]}\Big)\Big\}-\frac23,
\yestag\label{eq:xi}
\end{align*}
and $F_{Y^{(2)}}$ is the cumulative distribution function for $Y^{(2)}$. One readily verifies $\lvert \Xi_{[i]}\rvert\le 1$. 

Using \eqref{eq:Angus}, the limiting null distribution of $(-n^{1/2}\xi_n/3,\Lambda_n)$ will be the same as that of \\$(n^{-1/2}\sum_{i=1}^{n-1}\Xi_{[i]},\Lambda_n)$. 
To find the limiting null distribution of $(n^{-1/2}\sum_{i=1}^{n-1}\Xi_{[i]},\Lambda_n)$, using the idea from \citet[p.~210--214]{MR0229351}, we first find the limiting null distribution of 
\begin{align*}
\Big(n^{-1/2}\sum_{i=1}^{n-1}\Xi_{[i]},T_n\Big)
&=\Big(n^{-1/2}\sum_{i=1}^{n-1}\Xi_{[i]},n^{-1/2}\Delta_0\sum_{i=1}^{n}L'(\mY_i;0)\Big)\\
&=\Big(n^{-1/2}\sum_{i=1}^{n-1}\Xi_{[i]},n^{-1/2}\Delta_0\sum_{i=1}^{n}L'(\mY_{[i]};0)\Big),
\end{align*}
where $\mY_{[i]}\equiv (\mY^{(1)}_{[i]},\mY^{(2)}_{[i]})$. 
%Next we derive the joint distribution of 
%\[n^{-1/2}\sum_{i=1}^{n-1}\Xi_{[i]}~~~\text{and}~~~T_n=n^{-1/2}\Delta_0\sum_{i=1}^{n}L'(\mY_{[i]};0).\]
To employ the Cram\'er--Wold device, we aim to show that under the null, for any real numbers $a$ and $b$, 
\[
an^{-1/2}\sum_{i=1}^{n-1}\Xi_{[i]}+ bn^{-1/2}\Delta_0\sum_{i=1}^{n}L'(\mY_{[i]};0)
\to N\Big(0,2a^2/45+b^2\Delta_0^2\cI_{\mX}(0)\Big) ~~~\text{in distribution}.
\yestag\label{eq:cramerwold1}
\]
The idea of the proof is to first show a conditional central limit result
\begin{align*}
an^{-1/2}\sum_{i=1}^{n-1}\Xi_{[i]}+ bn^{-1/2}\Delta_0\sum_{i=1}^{n}L'(\mY_{[i]};0)\Big\vert Y^{(1)}_{1},\dots,Y^{(1)}_{n}
\to N\Big(0,2a^2/45+b^2\Delta_0^2\cI_{\mX}(0)\Big)\\
~~~\text{in distribution, for almost every sequence $Y^{(1)}_{1},\dots,Y^{(1)}_{n},\dots$},
\yestag\label{eq:condCLT}
\end{align*}
and secondly deduce the desired unconditional central limit result. 

To prove \eqref{eq:condCLT}, we follow the idea put forward in the proof of Lemma~2.9.5 in \citet{MR1385671}. 
According to the central limit theorem for $1$-dependent random
variables (see, e.g., the Corollary in \citealp[p.~546]{MR97841}), 
the statement \eqref{eq:condCLT} is true if the following conditions hold: for almost every sequence $Y^{(1)}_{1},\dots,Y^{(1)}_{n},\dots$, 
\begin{align*}
&\E^{(2)}\Big( W_{[i]} \Big)= 0, 
\yestag\label{eq:Lindeberg0}\\
&\frac{1}{n}\E^{(2)}\Big\{\Big(\sum_{i=1}^{n}
  W_{[i]} \Big)^2\Big\}\to 2a^2/45+b^2\Delta_0^2\cI_{\mX}(0), 
\yestag\label{eq:Lindeberg1}\\
&\sum_{i=1}^{n}\E^{(2)}\Big( W_{[i]} ^2\Big)\Big/\E^{(2)}\Big\{\Big(\sum_{i=1}^{n}
  W_{[i]} \Big)^2\Big\}~~~\text{is bounded},
\yestag\label{eq:Lindeberg3}\\
\text{and}~~~
&\frac{1}{n}\sum_{i=1}^{n}
\E^{(2)} \Big\{ W_{[i]} ^2
\times \ind\Big( n^{-1/2}\Big\lvert W_{[i]} \Big\rvert > \epsilon \Big)\Big\}
\to 0
~~~\text{for every $\epsilon>0$},
\yestag\label{eq:Lindeberg2}
\end{align*}
where $\E^{(2)}$ denotes the expectation conditionally on $Y^{(1)}_{1},\dots,Y^{(1)}_{n}$, and
\[
W_{[i]}\equiv a\Xi_{[i]}
 +b\Delta_0L'\Big(Y_{[i]};0\Big)
~~~\text{for}~i=1,\dots,n-1,
~~~\text{and}~~~
W_{[n]}\equiv 
  b\Delta_0L'\Big(Y_{[n]};0\Big).
\yestag\label{eq:defW}
\]

We verify conditions \eqref{eq:Lindeberg0}--\eqref{eq:Lindeberg2} as follows, starting from \eqref{eq:Lindeberg0}.
Under the null hypothesis, conditionally on $Y^{(1)}_{1},\dots,Y^{(1)}_{n}$, we have that $Y^{(2)}_{[1]},\dots,Y^{(2)}_{[n]}$ are still independent and identically distributed as $Y^{(2)}$, which implies that $\E^{(2)}(\Xi_{[i]})=0$. 
We also deduce, by \eqref{eq:whynontrivial} and Assumption~\ref{asp:kon}\ref{asp:kon2}, that 
\[
\E \Big\{L'\Big(Y;0\Big)\Big\vert Y^{(1)}\Big\}=0,
\yestag\label{eq:assumpmixed}
\]
and thus $\E^{(2)}\{L'(Y_{[i]};0)\}=0$. 
Then \eqref{eq:Lindeberg0} follows by noticing that
\[
\E^{(2)}(\Xi_{[i]})=0%$\E[L'(Y;0)|Y^{(1)}]=0$.  
~~~\text{and}~~~
\E^{(2)}\{L'(Y_{[i]};0)\}=0.
\yestag\label{eq:start}
\] 

For \eqref{eq:Lindeberg1} and \eqref{eq:Lindeberg3}, we first claim that 
\[
\Cov^{(2)}\Big\{n^{-1/2}\sum_{i=1}^{n-1}\Xi_{[i]},n^{-1/2}\Delta_0\sum_{i=1}^{n}L'(\mY_{[i]};0)\Big)\Big\}=0,
\yestag\label{eq:cov0}
\] 
where $\cov^{(2)}$ denotes the covariance conditionally on $Y^{(1)}_{1},\dots,Y^{(1)}_{n}$. 
Recall that, under the null hypothesis, $Y^{(2)}_{[1]},\dots,Y^{(2)}_{[n]}$ are still independent and identically distributed as $Y^{(2)}$, conditionally on $Y^{(1)}_{1},\dots,Y^{(1)}_{n}$.
We obtain
\begin{align*}
&\Cov^{(2)}\Big\{\Big|F_{Y^{(2)}}\Big(Y^{(2)}_{[i+1]}\Big)-F_{Y^{(2)}}\Big(Y^{(2)}_{[i]}\Big)\Big|, L'\Big(Y_{[i+1]};0\Big)\Big\}\\
=\;&\Cov^{(2)}\Big[\frac12\Big\{F_{Y^{(2)}}\Big(Y^{(2)}_{[i+1]}\Big)\Big\}^2
+\frac12\Big\{1-F_{Y^{(2)}}\Big(Y^{(2)}_{[i+1]}\Big)\Big\}^2, L'\Big(Y_{[i+1]};0\Big)\Big]\yestag\label{eq:ta1}
\end{align*}
by taking expectation with respect to $Y^{(2)}_{[i]}$,
\begin{align*}
&\Cov^{(2)}\Big\{\Big|F_{Y^{(2)}}\Big(Y^{(2)}_{[i+1]}\Big)-F_{Y^{(2)}}\Big(Y^{(2)}_{[i]}\Big)\Big|, L'\Big(Y_{[i]};0\Big)\Big\}\\
=\;&\Cov^{(2)}\Big[\frac12\Big\{F_{Y^{(2)}}\Big(Y^{(2)}_{[i]}\Big)\Big\}^2
+\frac12\Big\{1-F_{Y^{(2)}}\Big(Y^{(2)}_{[i]}\Big)\Big\}^2, L'\Big(Y_{[i]};0\Big)\Big]\yestag\label{eq:ta2}
\end{align*}
by taking expectation with respect to $Y^{(2)}_{[i+1]}$, and
\begin{align*}
&\Cov^{(2)}\Big\{\Big|F_{Y^{(2)}}\Big(Y^{(2)}_{[i+1]}\Big)-F_{Y^{(2)}}\Big(Y^{(2)}_{[i]}\Big)\Big|, L'\Big(Y_{[j]};0\Big)\Big\}=0~~~\text{for all}~j\ne i,~i+1,\yestag\label{eq:ta3}
\end{align*}
since $Y^{(2)}_{[i]},Y^{(2)}_{[i+1]}$ are independent of $Y^{(2)}_{[j]}$ with $j\ne i,~i+1$, conditionally on $Y^{(1)}_{1},\dots,Y^{(1)}_{n}$. 
Taking into account \eqref{eq:ta1}--\eqref{eq:ta3}, it follows that
\begin{align*}
&\Cov^{(2)}\Big\{n^{-1/2}\sum_{i=1}^{n-1}\Xi_{[i]},n^{-1/2}\Delta_0\sum_{i=1}^{n}L'\Big(Y_{[i]}\Big)\Big\}\\
=\;&n^{-1}\Delta_0\Big(\sum_{i=2}^{n}\Cov^{(2)}\Big[\frac12\Big\{F_{Y^{(2)}}\Big(Y^{(2)}_{[i]}\Big)\Big\}^2
+\frac12\Big\{1-F_{Y^{(2)}}\Big(Y^{(2)}_{[i]}\Big)\Big\}^2, L'\Big(Y_{[i]};0\Big)\Big]\\
&\qquad\quad+\sum_{i=1}^{n-1}\Cov^{(2)}\Big[\frac12\Big\{F_{Y^{(2)}}\Big(Y^{(2)}_{[i]}\Big)\Big\}^2
+\frac12\Big\{1-F_{Y^{(2)}}\Big(Y^{(2)}_{[i]}\Big)\Big\}^2, L'\Big(Y_{[i]};0\Big)\Big]\\
&\qquad\quad+\sum_{i=2}^{n}\Cov^{(2)}\Big[F_{Y^{(2)}}\Big(Y^{(2)}_{[i]}\Big)\Big\{1-F_{Y^{(2)}}\Big(Y^{(2)}_{[i]}\Big)\Big\}, L'\Big(Y_{[i]};0\Big)\Big]\\%\displaybreak[0]
&\qquad\quad+\sum_{i=1}^{n-1}\Cov^{(2)}\Big[F_{Y^{(2)}}\Big(Y^{(2)}_{[i]}\Big)\Big\{1-F_{Y^{(2)}}\Big(Y^{(2)}_{[i]}\Big)\Big\}, L'\Big(Y_{[i]};0\Big)\Big]\Big)\\
=\;&n^{-1}\Big[\sum_{i=2}^{n}\Cov^{(2)}\Big\{\frac12, L'\Big(Y_{[i]};0\Big)\Big\}+\sum_{i=1}^{n-1}\Cov^{(2)}\Big\{\frac12, L'\Big(Y_{[i]};0\Big)\Big\}\Big]\\
=\;&n^{-1}\Big[-\Cov^{(2)}\Big\{\frac12, L'\Big(Y_{[1]};0\Big)\Big\}-\Cov^{(2)}\Big\{\frac12, L'\Big(Y_{[n]};0\Big)\Big\}\Big] = 0,
%=\;&0,
\yestag\label{eq:jiewei1}
\end{align*}
where 
%the second last step holds due to
%\[
% n^{-1}\sum_{i=1}^{n}\E^{(2)} \Big\{L'\Big(Y_{[i]};0\Big)\Big\}
%=n^{-1}\sum_{i=1}^{n}\E^{(2)} \Big\{L'\Big(Y_{i};0\Big)\Big\}
%=\E \Big\{L'\Big(Y;0\Big)\Big\vert Y^{(1)}\Big\}=0
%\yestag\label{eq:jiewei2}
%\]
%by \eqref{eq:assumpmixed}; 
we notice that
\[
   \E \Big\{n^{-1}\Big\lvert L'\Big(Y_{[j]};0\Big)\Big\rvert\Big\} 
\le\E \Big\{n^{-1}\sum_{i=1}^{n}\Big\lvert L'\Big(Y_{i};0\Big)\Big\rvert\Big\}
=  \E \Big\lvert L'\Big(Y;0\Big)\Big\rvert
<  \infty,
\yestag\label{eq:jiewei2}
\] 
for any given $j$.
%and the last step holds due to \eqref{eq:formofLprime}. 
%\[
%   \E^{(2)} \Big\{n^{-1}\Big\lvert L'\Big(Y_{[j]};0\Big)\Big\rvert\Big\} 
%\le\E^{(2)} \Big\{n^{-1}\sum_{i=1}^{n}\Big\lvert L'\Big(Y_{i};0\Big)\Big\rvert\Big\}
%=  \E^{(2)} \Big\lvert L'\Big(Y;0\Big)\Big\rvert
%<  \infty
%\] 
%for any given $j$.
%Furthermore, recall that
%\[T_n= n^{-1/2}\Delta_0
%\sum_{i=1}^{n}\Big[-Y^{(1)}_i\Big\{f'_{Y^{(2)}}\Big(Y^{(2)}_i\Big)\Big/f_{Y^{(2)}}\Big(Y^{(2)}_i\Big)\Big\}
%     -Y^{(2)}_i\Big\{f'_{Y^{(1)}}\Big(Y^{(1)}_i\Big)\Big/f_{Y^{(1)}}\Big(Y^{(1)}_i\Big)\Big\}\Big].
%\]
%We thus have, under the null, 
%\[\Cov\Big(n^{-1/2}\sum_{i=1}^{n-1}\Xi_i,T_n\Big)=0.\]
Then using \eqref{eq:start}--\eqref{eq:cov0} we can prove \eqref{eq:Lindeberg1} as follows:
\begin{align*}
\frac{1}{n}\E^{(2)}\Big\{\Big(\sum_{i=1}^{n}W_{[i]}&\Big)^2\Big\}
=\frac{1}{n}\E^{(2)}\Big[\Big\{a\sum_{i=1}^{n-1}\Xi_{[i]}+ b\Delta_0\sum_{i=1}^{n}L'\Big(\mY_{[i]};0\Big)\Big\}^2\Big]\\
=\;&\frac{1}{n}\E^{(2)}\Big[\Big(a\sum_{i=1}^{n-1}\Xi_{[i]}\Big)^2+ \Big\{b\Delta_0\sum_{i=1}^{n}L'\Big(\mY_{[i]};0\Big)\Big\}^2\Big]\\
=\;&\frac{1}{n}\E^{(2)}\Big[\Big(a\sum_{i=1}^{n-1}\Xi_{[i]}\Big)^2+ \Big\{b\Delta_0\sum_{i=1}^{n}L'\Big(\mY_{i};0\Big)\Big\}^2\Big]\\
=\;&\frac{2a^2(n-1)}{45n}+ \frac{1}{n}\sum_{i=1}^{n}\E^{(2)}\Big[\Big\{b\Delta_0L'\Big(\mY_{i};0\Big)\Big\}^2\Big]
\to2a^2/45+b^2\Delta_0^2\cI_{\mX}(0),
\yestag\label{eq:cauchy1}
\end{align*}
where the last step holds for almost all sequences $Y^{(1)}_{1},\dots,Y^{(1)}_{n},\dots$ by the law of large numbers. 

%To verify \eqref{eq:Lindeberg3}, using the Cauchy--Schwarz inequality, we obtain
%\begin{align*}
%\frac1n\sum_{i=1}^{n}\E^{(2)}\Big( W_{[i]} ^2\Big)
%&=\frac1n\bigg(\sum_{i=1}^{n-1}\E^{(2)}\Big[\Big\{a\Xi_{[i]}
% +b\Delta_0L'\Big(Y_{[i]};0\Big)\Big\}^2\Big]
%+\E^{(2)}\Big[\Big\{b\Delta_0L'\Big(Y_{[n]};0\Big)\Big\}^2\Big]\bigg)\\
%&\le\frac1n\bigg(2\sum_{i=1}^{n-1}\E^{(2)}\Big\{\Big(a\Xi_{[i]}\Big)^2\Big\}
%+2\sum_{i=1}^{n}\E^{(2)}\Big[\Big\{b\Delta_0L'\Big(Y_{[i]};0\Big)\Big\}^2\Big]\bigg)\\
%&=\frac1n\bigg(\frac{4a^2(n-1)}{45}
%+2\sum_{i=1}^{n}\E^{(2)}\Big[\Big\{b\Delta_0L'\Big(Y_{i};0\Big)\Big\}^2\Big]\bigg).
%\end{align*}
%Hence we have, recalling \eqref{eq:cauchy1},
%\[
%\sum_{i=1}^{n}\E^{(2)}\Big( W_{[i]} ^2\Big)\Big/\E^{(2)}\Big\{\Big(\sum_{i=1}^{n}
%  W_{[i]} \Big)^2\Big\}\le 2.
%\yestag\label{eq:end}
%\]

To verify \eqref{eq:Lindeberg3}, recalling \eqref{eq:xi} and using \eqref{eq:start},\eqref{eq:ta2}, we obtain
\[\E^{(2)}\Big\{\Xi_{[i]}\times\Delta_0L'\Big(Y_{[i]};0\Big)\Big\}
=\Cov^{(2)}\Big\{\Xi_{[i]},\Delta_0L'\Big(Y_{[i]};0\Big)\Big\}=0,\]
and moreover,
\begin{align*}
\frac1n\sum_{i=1}^{n}\E^{(2)}\Big( W_{[i]} ^2\Big)
&=\frac1n\bigg(\sum_{i=1}^{n-1}\E^{(2)}\Big[\Big\{a\Xi_{[i]}
 +b\Delta_0L'\Big(Y_{[i]};0\Big)\Big\}^2\Big]
+\E^{(2)}\Big[\Big\{b\Delta_0L'\Big(Y_{[n]};0\Big)\Big\}^2\Big]\bigg)\\
&=\frac1n\bigg(\sum_{i=1}^{n-1}\E^{(2)}\Big\{\Big(a\Xi_{[i]}\Big)^2\Big\}
+\sum_{i=1}^{n}\E^{(2)}\Big[\Big\{b\Delta_0L'\Big(Y_{[i]};0\Big)\Big\}^2\Big]\bigg)\\
&=\frac1n\bigg(\frac{2a^2(n-1)}{45}
+\sum_{i=1}^{n}\E^{(2)}\Big[\Big\{b\Delta_0L'\Big(Y_{i};0\Big)\Big\}^2\Big]\bigg).
\end{align*}
Hence we have, recalling \eqref{eq:cauchy1},
\[
\sum_{i=1}^{n}\E^{(2)}\Big( W_{[i]} ^2\Big)\Big/\E^{(2)}\Big\{\Big(\sum_{i=1}^{n}
  W_{[i]} \Big)^2\Big\}=1.
\yestag\label{eq:end}
\]

For proving \eqref{eq:Lindeberg2}, we recall that as given in \eqref{eq:whynontrivial}
\[
L'\Big(Y_{[i]};0\Big)=
-Y_{[i]}^{(1)}\Big\{\rho_{2}\Big(Y_{[i]}^{(2)}\Big)\Big\}
-Y_{[i]}^{(2)}\Big\{\rho_{1}\Big(Y_{[i]}^{(1)}\Big)\Big\},
\yestag\label{eq:formofLprime}
\]
where $\rho_{k}(z)\equiv f'_{k}(z)/f_{k}(z)$. 
The existence of finite second moments assumed in Assumption~\ref{asp:kon}\ref{asp:kon3}, $\E\{(Y^{(1)})^2\} <\infty$ and $\E[\{\rho_{1}(Y^{(1)})\}^2] <\infty$, implies that
\[
\max_{1\le i\le n}n^{-1/2}\Big\lvert Y^{(1)}_{i}\Big\rvert\to 0
~~~\text{and}~~~ 
\max_{1\le i\le n}n^{-1/2}\Big\lvert \rho_{1}\Big(Y^{(1)}_{i}\Big) \Big\rvert\to 0
\yestag\label{eq:docov}
\]
for almost all sequences $Y^{(1)}_{1},\dots,Y^{(1)}_{n},\dots$ \citep[Theorem~5.2]{MR150889}. Since $\lvert \Xi_{[i]}\rvert\le 1$, we have
\begin{align*}
\ind\Big( n^{-1/2}\big\lvert W_{[i]} \big\rvert > \epsilon \Big)
\le \ind\Big( \lvert a\rvert n^{-1/2} > \epsilon/3 \Big)
&+  \ind\Big\{ \lvert b\rvert\times\Big(\max_{1\le i\le n}n^{-1/2}\Big\lvert Y^{(1)}_{i}\Big\rvert\Big)
                 \times\Big\lvert \rho_{2}\Big(Y_{[i]}^{(2)}\Big)\Big\rvert > \epsilon/3 \Big\}\\
&+  \ind\Big\{ \lvert b\rvert\times\Big(\max_{1\le i\le n}n^{-1/2}\Big\lvert \rho_{1}\Big(Y^{(1)}_{i}\Big) \Big\rvert\Big)
                 \times\Big\lvert Y_{[i]}^{(2)}\Big\rvert > \epsilon/3 \Big\}.
\end{align*}
Then for every $\epsilon>0$, 
\begin{align*}
&\frac{1}{n}\sum_{i=1}^{n}
\E^{(2)} \Big\{ W_{[i]} ^2
\times \ind\Big( n^{-1/2}\Big\lvert W_{[i]} \Big\rvert > \epsilon \Big)\Big\}\\
\le\;
&\frac{1}{n}\sum_{i=1}^{n}
\E^{(2)} \bigg( 3\Big[a^2 +
\Big\{Y_{[i]}^{(1)}\rho_{2}\Big(Y_{[i]}^{(2)}\Big)\Big\}^2 +
\Big\{Y_{[i]}^{(2)}\rho_{1}\Big(Y_{[i]}^{(1)}\Big)\Big\}^2\Big]\\
&\qquad\,\times\Big[
   \ind\Big( \lvert a\rvert n^{-1/2} > \epsilon/3 \Big)+  \ind\Big\{ \lvert b\rvert\times\Big(\max_{1\le i\le n}n^{-1/2}\Big\lvert Y^{(1)}_{i}\Big\rvert\Big)
                 \times\Big\lvert \rho_{2}\Big(Y_{[i]}^{(2)}\Big)\Big\rvert > \epsilon/3 \Big\}\\
&\qquad\qquad\qquad\qquad\qquad\qquad+  \ind\Big\{ \lvert b\rvert\times\Big(\max_{1\le i\le n}n^{-1/2}\Big\lvert \rho_{1}\Big(Y^{(1)}_{i}\Big) \Big\rvert\Big)
                 \times\Big\lvert Y_{[i]}^{(2)}\Big\rvert > \epsilon/3 \Big\}\Big]\bigg). \yestag\label{eq:zuihou}
\end{align*}
Here in \eqref{eq:zuihou} we have by \eqref{eq:docov} and dominated convergence theorem that
\begin{align*}
&\frac{1}{n}\sum_{i=1}^{n}\E^{(2)} 
\Big[\ind\Big\{ \lvert b\rvert\times\Big(\max_{1\le i\le n}n^{-1/2}\Big\lvert Y^{(1)}_{i}\Big\rvert\Big)
                 \times\Big\lvert \rho_{2}\Big(Y_{[i]}^{(2)}\Big)\Big\rvert > \epsilon/3 \Big\}\Big]\\
=\;&\E^{(2)} 
\Big[\ind\Big\{ \Big\lvert b\Big\rvert\times\Big(\max_{1\le i\le n}n^{-1/2}\Big\lvert Y^{(1)}_{i}\Big\rvert\Big)
                 \times\Big\lvert \rho_{2}\Big(Y_{1}^{(2)}\Big)\Big\rvert > \epsilon/3 \Big\}\Big]\to 0,
\end{align*}
where
\[
\ind\Big\{ \lvert b\rvert\times\Big(\max_{1\le i\le n}n^{-1/2}\Big\lvert Y^{(1)}_{i}\Big\rvert\Big)
                 \times\Big\lvert \rho_{2}\Big(Y_{1}^{(2)}\Big)\Big\rvert > \epsilon/3 \Big\}\Big]\to 0 ~~~\text{in probability},
\]
for almost all sequences $Y^{(1)}_{1},\dots,Y^{(1)}_{n},\dots$. We also have
\begin{align*}       
&\frac{1}{n}\sum_{i=1}^{n}\E^{(2)} 
\Big[\Big\{Y_{[i]}^{(1)}\rho_{2}\Big(Y_{[i]}^{(2)}\Big)\Big\}^2 
\times\ind\Big\{ \lvert b\rvert\times\Big(\max_{1\le i\le n}n^{-1/2}\Big\lvert Y^{(1)}_{i}\Big\rvert\Big)
                 \times\Big\lvert \rho_{2}\Big(Y_{[i]}^{(2)}\Big)\Big\rvert > \epsilon/3 \Big\}\Big]\\
=\;&\frac{1}{n}\sum_{i=1}^{n}\Big(Y_{[i]}^{(1)}\Big)^2\E^{(2)} 
\Big[\Big\{\rho_{2}\Big(Y_{[i]}^{(2)}\Big)\Big\}^2 
\times\ind\Big\{ \lvert b\rvert\times\Big(\max_{1\le i\le n}n^{-1/2}\Big\lvert Y^{(1)}_{i}\Big\rvert\Big)
                 \times\Big\lvert \rho_{2}\Big(Y_{[i]}^{(2)}\Big)\Big\rvert > \epsilon/3 \Big\}\Big]\\
=\;&\bigg(\frac{1}{n}\sum_{i=1}^{n}\Big(Y_{[i]}^{(1)}\Big)^2\bigg)\bigg(\E^{(2)} 
\Big[\Big\{\rho_{2}\Big(Y_{1}^{(2)}\Big)\Big\}^2 
\times\ind\Big\{ \lvert b\rvert\times\Big(\max_{1\le i\le n}n^{-1/2}\Big\lvert Y^{(1)}_{i}\Big\rvert\Big)
                 \times\Big\lvert \rho_{2}\Big(Y_{1}^{(2)}\Big)\Big\rvert > \epsilon/3 \Big\}\Big]\bigg)\\
=\;&\bigg(\frac{1}{n}\sum_{i=1}^{n}\Big(Y_{i}^{(1)}\Big)^2\bigg)\bigg(\E^{(2)} 
\Big[\Big\{\rho_{2}\Big(Y_{1}^{(2)}\Big)\Big\}^2 
\times\ind\Big\{ \lvert b\rvert\times\Big(\max_{1\le i\le n}n^{-1/2}\Big\lvert Y^{(1)}_{i}\Big\rvert\Big)
                 \times\Big\lvert \rho_{2}\Big(Y_{1}^{(2)}\Big)\Big\rvert > \epsilon/3 \Big\}\Big]\bigg)\\
\to\,&0,
\end{align*}
where for almost all sequences $Y^{(1)}_{1},\dots,Y^{(1)}_{n},\dots$,
\[
\frac{1}{n}\sum_{i=1}^{n}\Big(Y_{i}^{(1)}\Big)^2\to \E\Big\{\Big(Y^{(1)}\Big)^2\Big\}
\]
by the law of large numbers, and
\[
\E^{(2)} 
\Big[\Big\{\rho_{2}\Big(Y_{1}^{(2)}\Big)\Big\}^2 
\times\ind\Big\{ |b|\times\Big(\max_{1\le i\le n}n^{-1/2}\Big\lvert Y^{(1)}_{i}\Big\rvert\Big)
                 \times\Big\lvert \rho_{2}\Big(Y_{1}^{(2)}\Big)\Big\rvert > \epsilon/3 \Big\}\Big]\to 0
\]
by \eqref{eq:docov} and the dominated convergence theorem.  
We can deduce similar convergences for all the other summands in \eqref{eq:zuihou}. 
Hence for almost all sequences $Y^{(1)}_{1},\dots,Y^{(1)}_{n},\dots$, all conditions \eqref{eq:Lindeberg0}--\eqref{eq:Lindeberg2} are satisfied. This completes the proof of \eqref{eq:condCLT}. Moreover, the desired result \eqref{eq:cramerwold1} follows.

Finally, the Cram\'er--Wold device yields that under the null,
\[
\Big(n^{-1/2}\sum_{i=1}^{n-1}\Xi_{[i]},T_n\Big)
\to
N_{2}\bigg(\bigg(\begin{matrix}0\\0\end{matrix}\bigg),
\bigg(\begin{matrix}2/45 & 0\\
0 & \Delta_0^2\cI_{\mX}(0)\end{matrix}\bigg)\bigg) ~~~\text{in distribution}.
\yestag\label{eq:cramerwold2}
\]
Furthermore, using ideas from \citet[p.~210--214]{MR0229351}
(see also \citealp[Appx.~B]{MR2691505}), we have under the null, 
\[\Lambda_n-T_n+\Delta_0^2\cI_{\mX}(0)/2\to 0 ~~~\text{in probability},\]
and thus under the null, 
\[
\Big(n^{-1/2}\sum_{i=1}^{n-1}\Xi_{[i]},\Lambda_n\Big)
\to
N_{2}\bigg(\bigg(\begin{matrix}0\\-\Delta_0^2\cI_{\mX}(0)/2\end{matrix}\bigg),
\bigg(\begin{matrix}2/45 & 0\\
0 & \Delta_0^2\cI_{\mX}(0)\end{matrix}\bigg)\bigg) ~~~\text{in distribution},
\yestag\label{eq:cramerwold3}
\]
and $(-n^{1/2}\xi_n/3,\Lambda_n)$ has the same limiting null distribution by \eqref{eq:Angus}.  
Finally, we employ a corollary to Le~Cam's third lemma (\citealp[Example~6.7]{MR1652247}) to obtain that, under the considered local alternative $H_{1,n}(\Delta_0)$ with any fixed $\Delta_0>0$, $-n^{1/2}\xi_n/3\to N(0,2/45)$ in distribution, and thus
\[n^{1/2}\xi_n
\to
N(0,2/5) ~~~\text{in distribution}.
\yestag\label{eq:final}
\]
This completes the proof for family (A).

{\bf (B)} {\it This proof proceeds with only Assumption~\ref{asp:far}\ref{asp:far1},\ref{asp:far2},\ref{asp:far5}.}
Let $\mY_{i}=(Y^{(1)}_{i},Y^{(2)}_{i})$, $i=1,\dots,n$ be independent copies of $\mY$ (distributed as $\mX$ with $\Delta=0$).
%Let $F^{(0)}$ and $F^{(a)}$ be the (joint) distribution functions of $(\mY_{1},\dots,\mY_{n})$ and $(\mX_{1},\dots,\mX_{n})$, respectively. %, and define $\Lambda_n\equiv \log({\d F^{(a)}}/{\d F^{(0)}})$.
%Let $f_0$ and $g$ denote density functions of $F_0$ and $G$, respectively. 
Denote
\[
L(\mx;\Delta)\equiv \frac{f_{\mX}(\mx;\Delta)}{f_{\mX}(\mx;0)},~~~
L'(\mx;\Delta)\equiv \frac{\partial}{\partial\Delta}L(\mx;\Delta), 
\]
and define $\Lambda_n=\sum_{i=1}^{n}\log L(\mY_i;\Delta_n)$ and $T_n\equiv \Delta_n\sum_{i=1}^{n}L'(\mY_{i};0)$. Direct computation yields 
\[
L(\mx;\Delta)\equiv \frac{(1-\Delta)f_{0}(\mx)+\Delta g(\mx)}{f_0(\mx)},~~~
L'(\mx;0)=\frac{g(\mx)-f_0(\mx)}{f_0(\mx)},
\]
and thus
\begin{align*}
\cI_{\mX}(0)&=\E[\{L'(\mY;0)\}^2]=\E[\{g(Y)/f_0(Y)-1\}^2]\\
&=\E[\{s(Y)\}^2]=\chi^2(G,F_0)\equiv \int (\d G/\d F_0-1)^2 \d F_0.
\end{align*}
%Notice also 
%since $g(Y)/f_0(Y)=\{g_1(Y^{(1)})/f_1(Y^{(1)})\}\cdot\{g_2(Y^{(2)})/f_2(Y^{(2)})\}$, $\cI_{\mX}(0)<\infty$ implies that $\E[\{g_k(Y^{(k)})/f_k(Y^{(k)})\}^2]<\infty$ for $k=1,2$. 
%It holds by central limit theorem that $T_n$ is asymptotically normal with mean $0$ and variance $\Delta_0^2\cI_{\mX}(0)=\Delta_0^2\chi^2(G,F_0)$.

Similar to the proof for family (A), we proceed to determine the limiting
null distribution of $(-n^{1/2}\xi_n/3,\Lambda_n)$.  To this end, 
in view of the proof of Theorem 2 in \citet{MR3466185}, 
we first find the limiting null distribution of 
$(n^{-1/2}\sum_{i=1}^{n-1}\Xi_{[i]},T_n)$. 
The idea of deriving it is still to first show \eqref{eq:condCLT}, 
then \eqref{eq:cramerwold1}, and thus \eqref{eq:cramerwold2}.

Next we verify conditions \eqref{eq:Lindeberg0}--\eqref{eq:Lindeberg2}
for family (B). Notice that when we verify conditions
\eqref{eq:Lindeberg0}--\eqref{eq:Lindeberg3} for family (A) (from
\eqref{eq:start} to \eqref{eq:end}), we only use that 
\begin{enumerate}[label=(\arabic*)]
\item %the property that 
under the null hypothesis, $Y^{(2)}_{[1]},\dots,Y^{(2)}_{[n]}$ are still independent and identically distributed as $Y^{(2)}$, conditionally on $Y^{(1)}_{1},\dots,Y^{(1)}_{n}$, 
\item $\E\{L'(Y;0)\vert Y^{(1)}\}=0$, and
\item $0<\cI_{\mX}(0)<\infty$. 
\end{enumerate}
The first property always holds under the null hypothesis. 
The latter two are assumed or implied in Assumption~\ref{asp:far}\ref{asp:far2} and Assumption~\ref{asp:far}\ref{asp:far1},\ref{asp:far5}, respectively. 
Hence we can verify conditions \eqref{eq:Lindeberg0}--\eqref{eq:Lindeberg3} for family (B) using the same arguments. 
The only difference lies in proving \eqref{eq:Lindeberg2}. Since $s(x)=g(x)/f_0(x)-1$ is continuous and has compact support, it is bounded by some constant, say $C_s>0$. We have by definition of $W_{[i]}$ in \eqref{eq:defW}, 
\[\big\lvert W_{[i]} \big\rvert \le \lvert a\rvert + \lvert b\rvert \Delta_0 C_s,\]
and thus
\[\ind\Big( n^{-1/2}\Big\lvert W_{[i]} \Big\rvert > \epsilon \Big)=0
~~~\text{for all}~n>\Big(\frac{|a|+|b|\Delta_0 C_s}{\epsilon}\Big)^2.\]
Then \eqref{eq:Lindeberg2} follows by the dominated convergence theorem.

We have proven \eqref{eq:cramerwold2} for family (B). Furthermore, 
in the proof of Theorem 2 in \citet{MR3466185}, they showed that under the null,
\[
\Lambda_n-T_n+\Delta_0^2\cI_{\mX}(0)/2\to 0 ~~~\text{in probability}.
\yestag\label{eq:ddbasy}
\]
Thus under the null, we have \eqref{eq:cramerwold3} as well. The rest of the proof is 
to employ a corollary to Le~Cam's third lemma (\citealp[Example~6.7]{MR1652247}) 
to obtain \eqref{eq:final}. 
\end{proof}

\subsection{Proof of Theorem~\ref{thm:power}(ii)}

\begin{proof}[Proof of Theorem~\ref{thm:power}(ii)]
{\bf (A)} {\it This proof uses all of Assumption~\ref{asp:kon}.}
%This is an application of a corollary to Le~Cam's third lemma 
%(\citealp[Theorem~6.6]{MR1652247}). 
Let $\mY_{i}=(Y^{(1)}_{i},Y^{(2)}_{i})$ and $\mX_{i}=(X^{(1)}_{i},X^{(2)}_{i})$, $i=1,\dots,n$ be independent copies of $\mY$ and $\mX$, respectively.
Here $\mX$ depends on $n$ with $\Delta=\Delta_n=n^{-1/2}\Delta_0$. 
Let $F^{(0)}$ and $F^{(a)}$ be the (joint) distribution functions of $(\mY_{1},\dots,\mY_{n})$ and $(\mX_{1},\dots,\mX_{n})$, respectively. 
Denote
\[
L(\mx;\Delta)\equiv \frac{f_{\mX}(\mx;\Delta)}{f_{\mX}(\mx;0)},~~~
L'(\mx;\Delta)\equiv \frac{\partial}{\partial\Delta}L(\mx;\Delta), 
\]
and define $\Lambda_n=\sum_{i=1}^{n}\log L(\mY_i;\Delta_n)$ and $T_n\equiv \Delta_n\sum_{i=1}^{n}L'(\mY_{i};0)$. 
These definitions make sense by Assumption~\ref{asp:kon}\ref{asp:kon1},\ref{asp:kon2}. 

In this proof we will consider the Hoeffding decomposition of $\mu_n$ under the null:
\begin{equation}\label{eqn:Hdec-W}
 \mu_n
=\sum_{\ell=1}^{m^{\mu}}\underbrace{\binom{n}{\ell}^{-1}
  \sum_{1\le i_1<\cdots< i_{\ell}\le n}\binom{m^{\mu}}{\ell}\,\widetilde h^{\mu}_{\ell}\Big\{
  \Big(Y^{(1)}_{i_1},Y^{(2)}_{i_1}\Big),\dots,
  \Big(Y^{(1)}_{i_\ell},Y^{(2)}_{i_\ell}\Big)
  \Big\}}_{\displaystyle H^{\mu}_{n,\ell}}, 
\end{equation}
where
\begin{align*}
&\widetilde h^{\mu}_{\ell}(\my_1,\ldots,\my_{\ell}) 
\equiv h^{\mu}_{\ell}(\my_1,\ldots,\my_{\ell})-\E h^{\mu}-\sum_{k=1}^{\ell-1}\sum_{1\leq i_1<\cdots<i_k\leq\ell}\widetilde h^{\mu}_{k}(\my_{i_1},\ldots,\my_{i_k}),\\
&h^{\mu}_{\ell}(\my_1\ldots,\my_{\ell})\equiv \E h^{\mu}(\my_1\ldots,\my_{\ell},\mY_{\ell+1},\ldots,\mY_{m^{\mu}}),
~~~\E h^{\mu}\equiv \E h^{\mu}(\mY_1,\ldots,\mY_{m^{\mu}}), 
\end{align*}
and $\mY_1,\ldots,\mY_{m^{\mu}}$ are $m^{\mu}$ independent copies of $\mY$.  
Here $h^{\mu}$ is the ``symmetrized'' kernel 
and $m^{\mu}$ is the order of the kernel function $h^{\mu}$
for $\mu\in\{D,R,\tau^*\}$ related to \eqref{eq:Dn}, \eqref{eq:Rn}, or \eqref{eq:taustarn}:
\begin{align*}
h^{D}(\my_1,\dots,\my_5)
\equiv\;&\frac{1}{5!}\sum_{1\le i_1\ne \cdots\ne i_5\le 5}\frac14\\
&~~~\Big[\Big\{\ind\Big(y^{(1)}_{i_1}\leq y^{(1)}_{i_5}\Big)-\ind\Big(y^{(1)}_{i_2}\leq y^{(1)}_{i_5}\Big)\Big\}
         \Big\{\ind\Big(y^{(1)}_{i_3}\leq y^{(1)}_{i_5}\Big)-\ind\Big(y^{(1)}_{i_4}\leq y^{(1)}_{i_5}\Big)\Big\}\Big]\\%\displaybreak[0]
&~~~\Big[\Big\{\ind\Big(y^{(2)}_{i_1}\leq y^{(2)}_{i_5}\Big)-\ind\Big(y^{(2)}_{i_2}\leq y^{(2)}_{i_5}\Big)\Big\}
         \Big\{\ind\Big(y^{(2)}_{i_3}\leq y^{(2)}_{i_5}\Big)-\ind\Big(y^{(2)}_{i_4}\leq y^{(2)}_{i_5}\Big)\Big\}\Big],\\
h^{R}(\my_1,\dots,\my_6)
\equiv\;&\frac{1}{6!}\sum_{1\le i_1\ne \cdots\ne i_6\le 6}\frac14\\
&~~~\Big[\Big\{\ind\Big(y^{(1)}_{i_1}\leq y^{(1)}_{i_5}\Big)-\ind\Big(y^{(1)}_{i_2}\leq y^{(1)}_{i_5}\Big)\Big\}
         \Big\{\ind\Big(y^{(1)}_{i_3}\leq y^{(1)}_{i_5}\Big)-\ind\Big(y^{(1)}_{i_4}\leq y^{(1)}_{i_5}\Big)\Big\}\Big]\\
&~~~\Big[\Big\{\ind\Big(y^{(2)}_{i_1}\leq y^{(2)}_{i_6}\Big)-\ind\Big(y^{(2)}_{i_2}\leq y^{(2)}_{i_6}\Big)\Big\}
         \Big\{\ind\Big(y^{(2)}_{i_3}\leq y^{(2)}_{i_6}\Big)-\ind\Big(y^{(2)}_{i_4}\leq y^{(2)}_{i_6}\Big)\Big\}\Big],\\
h^{\tau^*}(\my_1,\dots,\my_4)
\equiv\;&\frac{1}{4!}\sum_{1\le i_1\ne \cdots\ne i_4\le 4}
    \Big\{\ind\Big(y^{(1)}_{i_1},y^{(1)}_{i_3}< y^{(1)}_{i_2},y^{(1)}_{i_4}\Big)
         +\ind\Big(y^{(1)}_{i_2},y^{(1)}_{i_4}< y^{(1)}_{i_1},y^{(1)}_{i_3}\Big)\\
&\mkern125mu
         -\ind\Big(y^{(1)}_{i_1},y^{(1)}_{i_4}< y^{(1)}_{i_2},y^{(1)}_{i_3}\Big)
         -\ind\Big(y^{(1)}_{i_2},y^{(1)}_{i_3}< y^{(1)}_{i_1},y^{(1)}_{i_4}\Big)\Big\}\\
&\mkern120mu
\Big\{\ind\Big(y^{(2)}_{i_1},y^{(2)}_{i_3}< y^{(2)}_{i_2},y^{(2)}_{i_4}\Big)
         +\ind\Big(y^{(2)}_{i_2},y^{(2)}_{i_4}< y^{(2)}_{i_1},y^{(2)}_{i_3}\Big)\\
&\mkern125mu
         -\ind\Big(y^{(2)}_{i_1},y^{(2)}_{i_4}< y^{(2)}_{i_2},y^{(2)}_{i_3}\Big)
         -\ind\Big(y^{(2)}_{i_2},y^{(2)}_{i_3}< y^{(2)}_{i_1},y^{(2)}_{i_4}\Big)\Big\},
\end{align*}
and $m^{D}=5$, $m^{R}=6$, $m^{\tau^*}=4$. 
We will omit the superscript $\mu$ in $m^{\mu}$, $h^{\mu}$,
$h^{\mu}_{\ell}$, $\widetilde h^{\mu}_{\ell}$, and $H^{\mu}_{n,\ell}$
hereafter if no confusion is possible.

The proof is split into three steps. 
First, we prove that $F^{(a)}$ is contiguous to $F^{(0)}$ in order to employ Le~Cam's third lemma \citep[Theorem~6.6]{MR1652247}.
Next, we find the limiting null distribution of $(n\mu_n, \Lambda_n)$. %\fbox{do not have to put $^\top$ in the new notation system?}. 
Lastly, we employ Le~Cam's third lemma to deduce the alternative distribution of $(n\mu_n, \Lambda_n)$.

{\bf Step I.} In view of \citet[Sec.~3.2.1]{MR2691505}, Assumption~\ref{asp:kon} is sufficient for the contiguity: we have that $F^{(a)}$ is contiguous to $F^{(0)}$.

{\bf Step II.} Next we need to derive the limiting distribution of $(n\mu_n, \Lambda_n)$ under null hypothesis. 
To this end, we first derive the limiting null distribution of $(nH_{n,2}, \Lambda_n)$, where $H_{n,2}$ is defined in \eqref{eqn:Hdec-W}. %at the beginning of Proof of Theorem~\ref{thm:smallop}(vi).
We write by the Fredholm theory of integral equations \citep[pages~1009,~1083,~1087]{MR0188745} that 
\[
H_{n,2}=
\frac{1}{n(n-1)}\sum_{i\ne j}
  \sum_{v=1}^{\infty}
  \lambda_{v}
  \psi_{v}\Big(Y^{(1)}_{i},Y^{(2)}_{i}\Big)
  \psi_{v}\Big(Y^{(1)}_{j},Y^{(2)}_{j}\Big), 
\] 
where $\{\lambda_{v},v=1,2,\dots\}$ is an arrangement of $\{\lambda_{v_1,v_2}, v_1,v_2=1,2,\dots\}$, and $\psi_{v}$ is the normalized eigenfunction associated with $\lambda_{v}$. 
For each positive integer $K$, define the ``truncated'' U-statistic as
\[
H_{n,2,K}\equiv 
\frac{1}{n(n-1)}\sum_{i\ne j}
  \sum_{v=1}^{K}
  \lambda_{v}
  \psi_{v}\Big(Y^{(1)}_{i},Y^{(2)}_{i}\Big)
  \psi_{v}\Big(Y^{(1)}_{j},Y^{(2)}_{j}\Big).
\]
%and derive the limiting distribution of $nH_{n,2,K}$ as $n\to\infty$.
Notice that $nH_{n,2}$ and $nH_{n,2,K}$ can be written as
\begin{align*}
n H_{n,2}&=\frac{n}{n-1}\Big(\sum_{v=1}^{\infty}\lambda_{v} \Big\{n^{-1/2}\sum_{i=1}^n \psi_{v}\Big(Y^{(1)}_{i},Y^{(2)}_{i}\Big)\Big\}^2
-\sum_{v=1}^{\infty}\lambda_{v} \Big[n^{-1}\sum_{i=1}^n \Big\{\psi_{v}\Big(Y^{(1)}_{i},Y^{(2)}_{i}\Big)\Big\}^2\Big]\Big),\\
n H_{n,2,K}&=\frac{n}{n-1}\Big(\sum_{v=1}^{K}\lambda_{v} \Big\{n^{-1/2}\sum_{i=1}^n \psi_{v}\Big(Y^{(1)}_{i},Y^{(2)}_{i}\Big)\Big\}^2
-\sum_{v=1}^{K}\lambda_{v} \Big[n^{-1}\sum_{i=1}^n \Big\{\psi_{v}\Big(Y^{(1)}_{i},Y^{(2)}_{i}\Big)\Big\}^2\Big]\Big).
%\yestag\label{eq:sqsum}
\end{align*}
For a simpler presentation, let $S_{n,v}$ denote $n^{-1/2}\sum_{i=1}^n {\psi_{v}(Y^{(1)}_{i},Y^{(2)}_{i})}$ 
%and let $\widetilde R_{n,k}$ denote $n^{-1}\sum_{i=1}^n \{\psi_{v}(Y^{(1)}_{i},Y^{(2)}_{i})\}^2$ 
hereafter.

To derive the limiting null distribution of $(nH_{n,2}, \Lambda_n)$, 
we first derive the limiting null distribution of $(nH_{n,2,K}, T_n)$ for each integer $K$.
Observe that
\begin{align*}
\E  (S_{n,v}) &= 0,~~~ 
\Var(S_{n,v})  = 1,~~~
\Cov(S_{n,v}, T_n)\to d_v\Delta_0,\\
\E  (T_n)     &= 0,~~~ 
\Var(T_n)      = \cI_{\mX}(0),~~~
\end{align*}
where $d_v\equiv \Cov\{\psi_{v}(\mY),L'(\mY;0)\}$
and $0<\cI_{X}(0)<\infty$ by Assumption~\ref{asp:kon}. 
There exists at least one $v\geq 1$ such that~$d_v\neq 0$.  
Indeed, applying Theorem~4.4 and Lemma~4.2 in \citet{MR3541972} yields 
\[
\Big\{\psi_{v}\Big(\mx\Big),v=1,2,\dots\Big\}
=\Big\{\psi_{v_1}^{(1)}\Big(x^{(1)}\Big)
       \psi_{v_2}^{(2)}\Big(x^{(2)}\Big),v_1,v_2=1,2,\dots\Big\},
\]
where
\[
\psi_{v_1}^{(1)}\Big(x^{(1)}\Big)\psi_{v_2}^{(2)}\Big(x^{(2)}\Big)
\equiv2\cos\Big\{\pi v_1 F_{Y^{(1)}}\Big(x^{(1)}\Big)\Big\}
       \cos\Big\{\pi v_2 F_{Y^{(2)}}\Big(x^{(2)}\Big)\Big\}
\]
is associated with eigenvalue $\lambda^{\mu}_{v_1,v_2}$ defined in Proposition~\ref{prop:cha}.
Since 
\[
\E Y^{(k)}=\E \Big\{\rho_{Y^{(k)}}\Big(Y^{(k)}\Big)\Big\}=0,
\] 
$\{\psi_{v}(\mx),v=1,2,\dots\}$ forms a complete orthogonal basis for the family of functions of the form \eqref{eq:whynontrivial}: $d_v=0$ for all $v$ thus entails 
\[
\cI_{X}(0)=\E[\{L^{\prime}(\mY;0)\}^2]=\E\Big[\Big\{\sum_{v=1}^{\infty}d_v\psi_{v}\Big(\mY^{(1)},\mY^{(2)}\Big)\Big\}^2\Big]=\sum_{v=1}^{\infty}d_v^2=0,
\]
which contradicts  Assumption~\ref{asp:kon}(iii). 
Therefore,  $d_{v^*}\ne0$ for some $v^*$. 
Applying the multivariate central limit theorem \citep[Equation~(18.24)]{MR855460}, 
we deduce that under the null,
\[
(S_{n,1}, \dots,S_{n,K}, T_n)
\to
(\xi_{1}, \dots,\xi_{K}, V_K)~~~\text{in distribution},
\]
where
\[
(\xi_{1}, \dots,\xi_{K}, V_K)
\sim N_{K+1}\bigg(\bigg(\begin{matrix}0_K\\ 0\end{matrix}\bigg),
\bigg(\begin{matrix}\fI_K & \Delta_0\mv\\
\Delta_0\mv^{\T} & \Delta_0^2\cI\end{matrix}\bigg)\bigg) . 
\]
Here $0_K$ denotes a zero vector of dimension $K$, $\fI_K$ denotes an identity matrix of dimension $K$, $\cI$ is short for $\cI_{\mX}(0)$, and $\mv=(d_1,\dots,d_K)$.
Thus $V_K$ can be expressed as 
\[
\Big(\Delta_0^2\cI\Big)^{1/2}\Big\{\sum_{v=1}^{K}c_v\xi_{v} + c_{0,K}\xi_0\Big\},
\]
where $c_v\equiv \cI^{-1/2}d_v$, $c_{0,K}\equiv (1-\sum_{v=1}^{K}c_v^2)^{1/2}$, 
and $\xi_0$ is standard Gaussian and independent of $\xi_{1},\dots,\xi_{K}$.
Then by the continuous mapping theorem \citep[Theorem~2.3]{MR1652247} 
and Slutsky's theorem \citep[Theorem~2.8]{MR1652247}, we have under the null,
\[
(nH_{n,2,K}, T_n)
\to
\bigg(\sum_{v=1}^{K}\lambda_v\Big(\xi_{v}^2-1\Big),~
\Big(\Delta_0^2\cI\Big)^{1/2}\Big(\sum_{v=1}^{K}c_v\xi_{v} + c_{0,K} \xi_0\Big)\bigg) ~~~\text{in distribution}.
\yestag\label{eq:jointKlimit}
\]
Moreover, we claim that under the null, 
\[
(nH_{n,2}, T_n)
\to
\bigg(\sum_{v=1}^{\infty}\lambda_v\Big(\xi_{v}^2-1\Big),~
\Big(\Delta_0^2\cI\Big)^{1/2}\Big(\sum_{v=1}^{\infty}c_v\xi_{v} + c_{0,\infty}\xi_0\Big)\bigg) ~~~\text{in distribution},
\yestag\label{eq:jointlimit}
\]
with $c_{0,\infty}\equiv (1-\sum_{v=1}^{\infty}c_v^2)^{1/2}$ via the following argument. 
Denote
\begin{align*}
M_K&\equiv  \sum_{v=1}^{K}\lambda_v\Big(\xi_{v}^2-1\Big), &  
V_K&\equiv  \Big(\Delta_0^2\cI\Big)^{1/2}\Big(\sum_{v=1}^{K}c_v \xi_{v} + c_{0,K} \xi_0\Big),\\
M&\equiv  \sum_{v=1}^{\infty}\lambda_v\Big(\xi_{v}^2-1\Big), &  \text{and}~~~
V&\equiv  \Big(\Delta_0^2\cI\Big)^{1/2}\Big(\sum_{v=1}^{\infty}c_v \xi_{v} + c_{0,\infty} \xi_0\Big).
\end{align*}
To prove \eqref{eq:jointlimit}, it suffices to prove that for any real numbers $a$ and $b$, 
\[
\Big\lvert\E\Big\{\exp\Big(\sfi anH_{n,2}+\sfi bT_n\Big)\Big\} -
\E\Big\{\exp\Big(\sfi aM+\sfi bV\Big)\Big\}\Big\rvert \to 0~~~\text{as}~n\to\infty,
\yestag\label{eq:charlimit}
\]
where $\sfi$ denotes the imaginary unit. 
We have
\begin{align*}
&\Big\lvert\E\Big\{\exp\Big(\sfi anH_{n,2}+\sfi bT_n\Big)\Big\} -
\E\Big\{\exp\Big(\sfi aM+\sfi bV\Big)\Big\}\Big\rvert \\
\le\;&\Big\lvert\E\Big\{\exp\Big(\sfi anH_{n,2}+\sfi bT_n\Big)\Big\} -
\E\Big\{\exp\Big(\sfi anH_{n,2,K}+\sfi bT_n\Big)\Big\}\Big\rvert \\
&+\Big\lvert\E\Big\{\exp\Big(\sfi anH_{n,2,K}+\sfi bT_n\Big)\Big\} -
\E\Big\{\exp\Big(\sfi aM_K+\sfi bV_K\Big)\Big\}\Big\rvert \\
&+\Big\lvert\E\Big\{\exp\Big(\sfi aM_K+\sfi bV_K\Big)\Big\} -
\E\Big\{\exp\Big(\sfi aM+\sfi bV\Big)\Big\}\Big\rvert\equiv I+I\!I+I\!I\!I,~~~\text{say,}
%\le\;&\E\Big\lvert e^{it_1n(H_{n,2}-H_{n,2,K})}-1\Big\rvert
%+\Big\lvert\E\Big[e^{it_1nH_{n,2,K}+it_2T_n}\Big] -
%\E\Big[e^{it_1M_K+it_2V_K}\Big]\Big\rvert 
%+\E\Big\lvert e^{it_1(M_K-M)+it_2(V_K-V)}-1\Big\rvert\\
%\le\;&\Big\{\E\Big\lvert t_1n(H_{n,2}-H_{n,2,K}\Big\rvert^2\Big\}^{1/2}
%+\Big\lvert\E\Big[e^{it_1nH_{n,2,K}+it_2T_n}\Big] -
%\E\Big[e^{it_1M_K+it_2V_K}\Big]\Big\rvert 
%+\Big\{\E\Big\lvert t_1(M_K-M)+t_2(V_K-V)\Big\rvert^2\Big\}^{1/2}\\
%\le\;&\Big\{\frac{2nt_1^2}{n-1}\sum_{v=K+1}^{\infty}\lambda_v^2\Big\}^{1/2}
%+\Big\lvert\E\Big[e^{it_1nH_{n,2,K}+it_2T_n}\Big] -
%\E\Big[e^{it_1M_K+it_2V_K}\Big]\Big\rvert 
%+\Big\{2\Big(2t_1^2\sum_{v=K+1}^{\infty}\lambda_v^2 + 
%2\Delta_0^2t_2^2\sum_{v=K+1}^{\infty}c_v^2\Big)\Big\}^{1/2}
\end{align*}
where in view of page 82 of \citet{MR1075417} and Equation~(4.3.10) in \citet{MR1472486}, 
\begin{align*}
I 
&\le\E\Big\lvert \exp\Big\{\sfi an\Big(H_{n,2}-H_{n,2,K}\Big)\Big\}-1\Big\rvert
\le\Big\{\E\Big\lvert an\Big(H_{n,2}-H_{n,2,K}\Big)\Big\rvert^2\Big\}^{1/2}
=\Big(\frac{2na^2}{n-1}\sum_{v=K+1}^{\infty}\lambda_v^2\Big)^{1/2}%,
\end{align*}
and
\begin{align*}
I\!I\!I 
&\le\E\Big\lvert \exp\Big\{\sfi a\Big(M_K-M\Big)+\sfi b\Big(V_K-V\Big)\Big\}-1\Big\rvert
\le\Big\{\E\Big\lvert a\Big(M_K-M\Big)+b\Big(V_K-V\Big)\Big\rvert^2\Big\}^{1/2}\\
&\le\Big\{2\Big(2a^2\sum_{v=K+1}^{\infty}\lambda_v^2 + 
2b^2\Delta_0^2\cI\sum_{v=K+1}^{\infty}c_v^2\Big)\Big\}^{1/2}.
\end{align*}
Since by Remark~3.1 in \citet{MR3541972}, 
\[
\sum_{v=1}^{\infty}\lambda_v^2=
\begin{cases}
{1}/{8100}  &\text{ when }\mu=D,R,\\
{1}/{225}   &\text{ when }\mu=\tau^*,
\end{cases}
~~~\text{and}~~~
\sum_{v=1}^{\infty}c_v^2=\cI^{-1}\sum_{v=1}^{\infty}d_v^2= 1, 
\]
we conclude that, for any $\epsilon>0$, there exists $K_0$ such that $I<\epsilon/3$ and $I\!I\!I<\epsilon/3$ for all $n$ and all $K\ge K_0$. For this $K_0$, we have $I\!I<\epsilon/3$ for all sufficiently large $n$ by \eqref{eq:jointKlimit}. These together prove \eqref{eq:charlimit}.
We also have, using the idea from \citet[p.~210--214]{MR0229351}
(see also \citealp[Appendix~B]{MR2691505}), that under the null
\[\Lambda_n-T_n+\Delta_0^2\cI/2\to 0 ~~~\text{in probability}.
\yestag\label{eq:firstapprox}\]
Combining \eqref{eq:jointlimit} and \eqref{eq:firstapprox} yields that under the null, 
\[
(nH_{n,2}, \Lambda_n)
\to
\bigg(\sum_{v=1}^{\infty}\lambda_v\Big(\xi_{v}^2-1\Big),~
\Big(\Delta_0^2\cI\Big)^{1/2}\Big(\sum_{v=1}^{\infty}c_v\xi_{v} + c_{0,\infty}\xi_0\Big)-\frac{\Delta_0^2\cI}{2}\bigg) ~~\text{in distribution}.
\yestag\label{eq:jointlimitL}
\]
Using the fact $H_{n,1}=0$ and Equation (1.6.7) in \citet[p.~30]{MR1075417} yields that $(n\mu_n, \Lambda_n)$ has the same limiting distribution as \eqref{eq:jointlimitL} under the null.

{\bf Step III.} %Noticing that
%$\widetilde H_{n,k}$, $k\in\Z_+$ are uniformly bounded by some absolute constant $b$ regardless of $k$ 
%since $\sup_v\lVert\psi_v\rVert<\infty$ \citep[Theorem~3.a.1]{MR0889455}, 
Finally employing Le Cam's third lemma \citep[Theorem~6.6]{MR1652247}
we obtain that under the local alternative
\begin{align*}
   \;& \pr\{n\mu_n\le q_{1-\alpha}\mid H_{1,n}(\Delta_0)\}\\
\to\;& \E \Big[\ind\Big\{\sum_{v=1}^{\infty}\lambda_v\Big(\xi_{v}^2-1\Big)\le q_{1-\alpha}\Big\} 
\times\exp\Big\{\Big(\Delta_0^2\cI\Big)^{1/2}\Big(\sum_{v=1}^{\infty}c_v \xi_{v} + c_{0,\infty} \xi_0\Big)-\frac{\Delta_0^2\cI}{2}\Big\}\Big]\\
\le\;& \E \Big[\ind\Big\{\Big\lvert \xi_{v^*}\Big\rvert\le \Big(\frac{q_{1-\alpha}+\sum_{v=1}^{\infty}\lambda_v}{\lambda_{v^*}}\Big)^{1/2}\Big\} 
\times\exp\Big\{\Big(\Delta_0^2\cI\Big)^{1/2}\Big(\sum_{v=1}^{\infty}c_v \xi_{v} + c_{0,\infty} \xi_0\Big)-\frac{\Delta_0^2\cI}{2}\Big\}\Big]\\
=\;& \E \Big[\ind\Big\{\Big\lvert \xi_{v^*}\Big\rvert\le \Big(\frac{q_{1-\alpha}+\sum_{v=1}^{\infty}\lambda_v}{\lambda_{v^*}}\Big)^{1/2}\Big\} 
\times\exp\Big\{\Big(\Delta_0^2\cI\Big)^{1/2}\Big(c_{v^*} \xi_{v^*} + (1-c_{v^*}^2)^{1/2} \xi_0\Big)-\frac{\Delta_0^2\cI}{2}\Big\}\Big]\\
=\;& \Phi\Big\{\Big(\frac{q_{1-\alpha}+\sum_{v=1}^{\infty}\lambda_v}{\lambda_{v^*}}\Big)^{1/2} -c_{v^*}\Big(\Delta_0^2\cI\Big)^{1/2}\Big\} 
- \Phi\Big\{-\Big(\frac{q_{1-\alpha}+\sum_{v=1}^{\infty}\lambda_v}{\lambda_{v^*}}\Big)^{1/2} -c_{v^*}\Big(\Delta_0^2\cI\Big)^{1/2}\Big\}\\
\le\;& 2\Big(\frac{q_{1-\alpha}+\sum_{v=1}^{\infty}\lambda_v}{\lambda_{v^*}}\Big)^{1/2} 
\varphi\Big\{\Big\lvert c_{v^*}\Big\rvert\Big(\Delta_0^2\cI\Big)^{1/2}-\Big(\frac{q_{1-\alpha}+\sum_{v=1}^{\infty}\lambda_v}{\lambda_{v^*}}\Big)^{1/2}\Big\},
\yestag\label{eq:linshi}
\end{align*}
for some $v^*$ such that  $c_{v^*}=\cI^{-1/2}d_{v^*}\ne0$ and
\[
\Delta_0 \ge \Big\lvert c_{v^*}\Big\rvert^{-1}\cI^{-1/2} \Big(\frac{q_{1-\alpha}+\sum_{v=1}^{\infty}\lambda_v}{\lambda_{v^*}}\Big)^{1/2},
\yestag\label{eq:bietai}
\]
where $\Phi$ and $\varphi$ are the distribution function and density function of the standard normal distribution, respectively. 
%The right-hand side of \eqref{eq:linshi} can be arbitrarily small as long as $\lvert\Delta_0\rvert$ is large enough, regardless of whether $c_{v^*}$ is positive or negative. 
Note that the right-hand side of \eqref{eq:linshi} is monotonically decreasing as $\Delta_0$ increases given \eqref{eq:bietai}. 
There exists a positive constant $C_\beta$ such that \eqref{eq:linshi} is smaller than $\beta/2$ as long as $\Delta_0\ge C_\beta$, regardless of whether $c_{v^*}$ is positive or negative. This concludes the proof.

{\bf (B)} {\it This proof uses Assumption~\ref{asp:far}\ref{asp:far1},\ref{asp:far4},\ref{asp:far5}.}
Let $\mY_{i}=(Y^{(1)}_{i},Y^{(2)}_{i})$, $i=1,\dots,n$ be independent copies of $\mY$ (distributed as $\mX$ with $\Delta=0$).
%Let $F^{(0)}$ and $F^{(a)}$ be the (joint) distribution functions of $(\mY_{1},\dots,\mY_{n})$ and $(\mX_{1},\dots,\mX_{n})$, respectively. %, and define $\Lambda_n\equiv \log({\d F^{(a)}}/{\d F^{(0)}})$.
%Let $f_0$ and $g$ denote density functions of $F_0$ and $G$, respectively. 
Denote
\[
L(\mx;\Delta)\equiv \frac{f_{\mX}(\mx;\Delta)}{f_{\mX}(\mx;0)},~~~
L'(\mx;\Delta)\equiv \frac{\partial}{\partial\Delta}L(\mx;\Delta), 
\]
and define $\Lambda_n=\sum_{i=1}^{n}\log L(\mY_i;\Delta_n)$ and $T_n\equiv \Delta_n\sum_{i=1}^{n}L'(\mY_{i};0)$. Direct computation yields 
\[
L(\mx;\Delta)\equiv \frac{(1-\Delta)f_{0}(\mx)+\Delta g(\mx)}{f_0(\mx)},~~~
L'(\mx;0)=\frac{g(\mx)-f_0(\mx)}{f_0(\mx)},
\]
and thus
\begin{align*}
\cI_{\mX}(0)&=\E[\{L'(\mY;0)\}^2]=\E[\{g(Y)/f_0(Y)-1\}^2]\\
&=\E[\{s(Y)\}^2]=\chi^2(G,F_0)\equiv \int (\d G/\d F_0-1)^2 \d F_0.
\end{align*}
%Notice also 
%since $g(Y)/f_0(Y)=\{g_1(Y^{(1)})/f_1(Y^{(1)})\}\cdot\{g_2(Y^{(2)})/f_2(Y^{(2)})\}$, $\cI_{\mX}(0)<\infty$ implies that $\E[\{g_k(Y^{(k)})/f_k(Y^{(k)})\}^2]<\infty$ for $k=1,2$. 
%It holds by central limit theorem that $T_n$ is asymptotically normal with mean $0$ and variance $\Delta_0^2\cI_{\mX}(0)=\Delta_0^2\chi^2(G,F_0)$.

This is similar to the proof for family (A). The only difference lies in
proving the existence of at least one $v\geq 1$ such that~$d_v\neq 0$, 
where $d_v\equiv \Cov[\psi_{v}(\mY),L'(\mY;0)]$. 
Now $L'(x;0)=s(x)$ is not of the form \eqref{eq:whynontrivial}, and 
$\{\psi_{v}(\mx),v=1,2,\dots\}$ does not necessarily form a complete orthogonal basis for the family of functions of $s(x)$. 
However, recall that
\[
\Big\{\psi_{v}\Big(\mx\Big),v=1,2,\dots\Big\}
=\Big\{\psi_{v_1}^{(1)}\Big(x^{(1)}\Big)
       \psi_{v_2}^{(2)}\Big(x^{(2)}\Big),v_1,v_2=1,2,\dots\Big\},
\]
where
\[
\psi_{v_1}^{(1)}\Big(x^{(1)}\Big)\psi_{v_2}^{(2)}\Big(x^{(2)}\Big)
\equiv2\cos\Big\{\pi v_1 F_{Y^{(1)}}\Big(x^{(1)}\Big)\Big\}
       \cos\Big\{\pi v_2 F_{Y^{(2)}}\Big(x^{(2)}\Big)\Big\}.
\]
Since
\[
\Big\{\psi_{v_1}^{(1)}\Big(x^{(1)}\Big)
      \psi_{v_2}^{(2)}\Big(x^{(2)}\Big),v_1,v_2=0,1,2,\dots\Big\}
\]
forms a complete orthogonal basis of the set of square integrable functions, $d_v=0$ for all $v\ge1$ thus entails $s(x)=h_1(x^{(1)})+h_2(x^{(2)})$ for some functions $h_1,h_2$, where $h_k(x^{(k)})$ depends only on $x^{(k)}$ for $k=1,2$. This contradicts Assumption~\ref{asp:far}\ref{asp:far4}. 
%
%
%
%
%
%
%
%{\bf (C)} {\it This proof uses Assumption~\ref{asp:gof}\ref{asp:gof1},\ref{asp:gof4},\ref{asp:gof5}.}
%The arguments are the same as for family (B) and hence omitted.  
%%by noticing $L'(\mx;0)=\Psi_1(x^{(1)})\Psi_2(x^{(2)})$ with 
%%$\E\{\Psi_1(X^{(1)})\}=\E\{\Psi_2(X^{(2)})\}=0$. 
\end{proof}

\subsection{Proof of Proposition~\ref{thm:opt}}

\begin{proof}[Proof of Proposition~\ref{thm:opt}]
{\bf (A)} {\it This proof uses all of Assumption~\ref{asp:kon}.}
Let $\mY_{i}=(Y^{(1)}_{i},Y^{(2)}_{i})$ and $\mX_{i}=(X^{(1)}_{i},X^{(2)}_{i})$, $i=1,\dots,n$ be independent copies of $\mY$ and $\mX$ with $\Delta=\Delta_n=n^{-1/2}\Delta_0$, respectively. 
Let $F^{(0)}$ and $F^{(a)}$ be the (joint) distribution functions of $(\mY_{1},\dots,\mY_{n})$ and $(\mX_{1},\dots,\mX_{n})$, respectively, 
and let $F^{(0)}_{i}$ and $F^{(a)}_{i}$ be the distribution functions of $\mY_{i}$ and $\mX_{i}$, respectively. 

The total variation distance between two distribution functions $G$ and $F$ on the same real probability space is defined as 
\[
\TV(G,F)\equiv\sup_{A}\Big|\pr_{G}(A)-\pr_{F}(A)\Big|,
\]
where $A$ is taken over the Borel field and $\pr_{G},\pr_{F}$ are respective probability measures induced by $G$ and $F$. Furthermore, if $G$ is absolutely continuous with respect to $F$, the Hellinger distance between $G$ and $F$ is defined as
\[
\HL(G,F)\equiv\Big[\int 2\Big\{1-({\d G}/{\d F})^{1/2}\Big\}\d F\Big]^{1/2}.
\]
By Assumption~\ref{asp:kon}\ref{asp:kon1}, $\HL(F^{(a)},F^{(0)})$ is well-defined. 
It suffices to prove that for any small $0<\beta<1-\alpha$, there exists $\Delta_0=c_\beta$ such that, for all sufficiently large $n$, $\TV(F^{(a)},F^{(0)})<\beta$, which is implied by 
$\HL(F^{(a)},F^{(0)})<\beta$ using the relation \citep[Equation~(2.20)]{MR2724359}
\[\TV\Big(F^{(a)},F^{(0)}\Big)\le\HL\Big(F^{(a)},F^{(0)}\Big).\]
We also know that \citep[page~83]{MR2724359}
\[
1-\frac{1}{2}\HL^2\Big(F^{(a)},F^{(0)}\Big)=\prod_{i=1}^{n}\Big\{1-\frac{1}{2}\HL^2\Big(F^{(a)}_{i},F^{(0)}_{i}\Big)\Big\}.
\]
We then aim to evaluate $\HL^2(F^{(a)},F^{(0)})$ in terms of $\cI_{\mX}(0)$ and $\Delta_0$. By definition,
\[
\frac12\HL^2\Big(F^{(a)}_{i},F^{(0)}_{i}\Big)
=\E\Big[1-\Big\{L\Big(\mY_{i};\Delta_n\Big)\Big\}^{1/2}\Big].
\]
Given Assumption~\ref{asp:kon}, we deduce in view of \citet[Appendix~B]{MR2691505} that
\[
n\E\Big[1-\Big\{L\Big(\mY_{i};\Delta_n\Big)\Big\}^{1/2}\Big]
=\E\Big(\sum_{i=1}^{n}\Big[1-\Big\{L\Big(\mY_{i};\Delta_n\Big)\Big\}^{1/2}\Big]\Big)
\to \frac{\Delta_0^2\cI_{\mX}(0)}{8}.
\]
Therefore, 
\[
1-\frac{1}{2}\HL^2\Big(F^{(a)},F^{(0)}\Big) \to \exp\Big\{-\frac{\Delta_0^2\cI_{\mX}(0)}{8}\Big\}.
\]
The desired result follows by taking $c_{\beta}>0$ such that
\[
\exp\Big\{-\frac{c_{\beta}^2\cI_{\mX}(0)}{8}\Big\}=1-\frac{\beta^2}{8}.
\]

{\bf (B)} {\it This proof requires Assumption~\ref{asp:far}\ref{asp:far1},\ref{asp:far5}.}
This is similar to the proof for family (A), but here we will use the relation \citep[Equation~(2.27)]{MR2724359}
\[
\TV\Big(F^{(a)},F^{(0)}\Big)\le\Big\{\chi^2\Big(F^{(a)},F^{(0)}\Big)\Big\}^{1/2},
\]
where the chi-square distance between two distribution functions $G$ and $F$ on the same real probability space such that $G$ is absolutely continuous with respect to $F$ is defined as 
\[
\chi^2(G,F)\equiv \int \Big(\d G/\d F-1\Big)^2 \d F.
\] 
Here $\chi^2(F^{(a)},F^{(0)})$ is well-defined by Assumption~\ref{asp:far}\ref{asp:far1}. 
We also know that \citep[page~86]{MR2724359}
\[
1+\chi^2\Big(F^{(a)},F^{(0)}\Big)
=\prod_{i=1}^{n}\Big\{1+\chi^2\Big(F^{(a)}_{i},F^{(0)}_{i}\Big)\Big\}.
\]
Next we aim to evaluate $\chi^2(F^{(a)},F^{(0)})$ in terms of
$\cI_{\mX}(0)=\chi^2(G,F_0)$ and $\Delta_0$. Here
$0<\cI_{\mX}(0)<\infty$ by
Assumption~\ref{asp:far}\ref{asp:far1},\ref{asp:far5}. We have by
definition that
\[
\chi^2\Big(F^{(a)}_{i},F^{(0)}_{i}\Big)=\chi^2\Big((1-\Delta_n)F_{0}+\Delta_n G, F_{0}\Big)=\Delta_n^2\chi^2(G,F_0)=n^{-1}\Delta_0^2\chi^2(G,F_0).
\]
Therefore, it holds that
\[
1+\chi^2\Big(F^{(a)},F^{(0)}\Big) \to \exp\Big\{\Delta_0^2\chi^2\Big(G,F_0\Big)\Big\}.
\]
The desired result follows by taking $c_{\beta}>0$ such that
\[
\exp\Big\{c_{\beta}^2\chi^2\Big(G,F_0\Big)\Big\}=1+\frac{\beta^2}{4}.
\]
%
%
%
%
%
%
%
%
%
%
%
%
%{\bf (C)} {\it This proof requires Assumption~\ref{asp:gof}\ref{asp:gof1},\ref{asp:gof5}.} 
%We omit any details as the arguments are exactly the same as in the proof for family (B).
%%This is also similar to the proof of Theorem~\ref{thm:opt-kon}, but here we will use 
%%using relation $\TV(F^{(a)},F^{(0)})\le\{\chi^2(F^{(a)},F^{(0)})\}^{1/2}$ \citep[Equation~(2.27)]{MR2724359}. We have also known that \citep[p.~86]{MR2724359}
%%\[
%%1+\chi^2(F^{(a)},F^{(0)})=\prod_{i=1}^{n}\Big(1+\chi^2(F^{(a)}_{i},F^{(0)}_{i})\Big).
%%\]
%%Next we aim to evaluate $\chi^2(\Q_{n},\P_{n})$ in terms of $\chi^2(G,F_0)$ and $\Delta_0$. We have by definition
%%\[
%%\chi^2(F^{(a)}_{i},F^{(0)}_{i})=\chi^2\Big((1-\Delta_n)F_{0}+\Delta_n G, F_{0}\Big)=\Delta_n^2\chi^2(G,F_0)=n^{-1}\Delta_0^2\chi^2(G,F_0).
%%\]
%%Therefore, we have
%%\[
%%1+\chi^2(F^{(a)},F^{(0)}) \to \exp\Big\{\Delta_0^2\chi^2(G,F_0)\Big\},
%%\]
%%and the result follows. 
This completes the proof.
\end{proof}

\section*{Acknowledgement}

We thank Holger Dette for pointing out related literature, and Sky Cao, Peter Bickel, and Bodhisattva Sen for many helpful comments. We thank the editor, the associate editor, and three anonymous referees for their very detailed and constructive comments and suggestions, which greatly helped improve the paper quality.
%This work was partially supported by NSF Grant SES-2019363 and xxx Grant XXX XXX.
This project has received funding support from the United States NSF Grant SES-2019363 
and European Research Council (ERC) 
under the European Union's Horizon 2020 research 
and innovation programme (grant agreement No 883818).

%\newpage
{%\small
\bibliographystyle{apalike}
\bibliography{AMS}
}

\end{document}